\documentclass[12pt]{amsart}
\usepackage[utf8]{inputenc}
\usepackage[shortlabels]{enumitem}
\usepackage{amssymb}
\usepackage{mathrsfs}
\usepackage{graphicx}
\usepackage{stmaryrd}
\usepackage[all]{xy}
\usepackage[margin=1in]{geometry} 
\usepackage[bookmarks, bookmarksdepth=2, colorlinks=true, linkcolor=blue, citecolor=blue, urlcolor=blue]{hyperref}
\usepackage{eucal}
\usepackage{contour}
\usepackage{shuffle}
\usepackage[normalem]{ulem}
\usepackage{tikz}
\usetikzlibrary{positioning}
\usetikzlibrary{arrows.meta}
\usetikzlibrary{decorations.markings}
\usetikzlibrary{plotmarks}
\usepackage{dsfont}

\numberwithin{equation}{section}
\newtheorem{theorem}[equation]{Theorem}

\newtheorem{proposition}[equation]{Proposition}
\newtheorem{lemma}[equation]{Lemma}
\newtheorem{corollary}[equation]{Corollary}

\newtheorem{problem}[equation]{Problem}
\newtheorem{question}[equation]{Question}
\theoremstyle{definition}
\newtheorem{rmk}[equation]{Remark}
\newenvironment{remark}[1][]{\begin{rmk}[#1] \pushQED{\qed}}{\popQED \end{rmk}}
\newtheorem{eg}[equation]{Example}
\newenvironment{example}[1][]{\begin{eg}[#1] \pushQED{\qed}}{\popQED \end{eg}}
\newtheorem{defnaux}[equation]{Definition}
\newenvironment{definition}[1][]{\begin{defnaux}[#1]\pushQED{\qed}}{\popQED \end{defnaux}}

\newcommand{\cB}{\mathcal{B}}

\newcommand{\rB}{\mathrm{B}}

\newcommand{\cC}{\mathcal{C}}

\newcommand{\sE}{\mathscr{E}}
\newcommand{\bF}{\mathbf{F}}

\newcommand{\cP}{\mathcal{P}}

\newcommand{\cR}{\mathcal{R}}

\newcommand{\bS}{\mathbf{S}}

\newcommand{\fS}{\mathfrak{S}}

\newcommand{\cT}{\mathcal{T}}

\newcommand{\bZ}{\mathbf{Z}}

\newcommand{\rf}{\mathrm{f}}

\newcommand{\arxiv}[1]{\href{http://arxiv.org/abs/#1}{{\tiny\tt arXiv:#1}}}
\newcommand{\DOI}[1]{\href{http://doi.org/#1}{\color{purple}{\tiny\tt DOI:#1}}}

\newcommand{\defn}[1]{\emph{#1}}

\let\ul\underline
\renewcommand{\phi}{\varphi}

\DeclareMathOperator{\im}{im} 
\DeclareMathOperator{\End}{End}
\DeclareMathOperator{\Alg}{Alg}
\DeclareMathOperator{\Et}{Et}

\DeclareMathOperator{\Mod}{Mod}
\DeclareMathOperator{\Fun}{Fun}
\DeclareMathOperator{\Hom}{Hom}
\DeclareMathOperator{\uHom}{\ul{Hom}}
\DeclareMathOperator{\Rep}{Rep}
\DeclareMathOperator{\Spec}{Spec}
\newcommand{\id}{\mathrm{id}}
\newcommand{\op}{\mathrm{op}}
\renewcommand{\Vec}{\mathrm{Vec}}
\newcommand{\GL}{\mathbf{GL}}

\DeclareMathOperator{\Mat}{Mat}
\DeclareMathOperator{\Sp}{Sp}

\newcommand{\bbone}{\mathds{1}}

\newcommand{\bzero}{\mathbf{0}}
\newcommand{\bone}{\mathbf{1}}
\newcommand{\uotimes}{\mathbin{\ul{\otimes}}}

\contourlength{1pt}

\contourlength{0.8pt}
\newcommand{\myuline}[1]{%
  \uline{\phantom{#1}}%
  \llap{\contour{white}{#1}}%
}
\DeclareMathOperator{\uRep}{\text{\myuline{\rm Rep}}}
\DeclareMathOperator{\uPerm}{\ul{Perm}}

\title{Discrete pre-Tannakian categories}

\author{Nate Harman}
\author{Andrew Snowden}

\date{April 11, 2023}

\begin{document}

\begin{abstract}
Pre-Tannakian categories are a natural class of tensor categories that can be viewed as generalizations of algebraic groups. We define a pre-Tannkian category to be \defn{discrete} if it is generated by an \'etale commutative algebra; these categories generalize finite groups. The main theorem of this paper establishes a rough classification of these categories: we show that any discrete pre-Tannakian category $\cC$ is associated to an oligomorphic group $G$, via a construction we recently introduced. In certain cases, such as when $\cC$ has enough projectives, we completely describe $\cC$ in terms of $G$.
\end{abstract}

\maketitle
\tableofcontents

\section{Introduction}

\subsection{Background}

The representation theory of algebraic groups plays an important role in many areas of mathematics. One way to view this theory is through the lens of tensor categories: the collection $\Rep(G)$ of finite dimensional representations of an algebraic group $G$ forms an abelian category that carries a natural tensor product. Taking the most important properties of these examples as axioms leads to the class of \defn{pre-Tannakian} tensor categories (see Definition~\ref{defn:pt}). We view these as fundamental objects that are intrinsically interesting. From this perspective, it is natural to consider the following problem:

\begin{problem} \label{prob:main}
Classify pre-Tannakian categories.
\end{problem}

In characteristic~0, Deligne \cite{Deligne2} characterized pre-Tannakian categories of the form $\Rep(G)$, with $G$ an algebraic (super)group. Subsequently, several authors have worked to extend this result to positive characteristic \cite{BE, BEO, Cou1, Cou2, CEO, CEOP, EO, EOV, Ostrik}. However, beyond this work, we are unaware of any systematic attack on Problem~\ref{prob:main}. This is not surprising, due to the dearth of examples: until recently, the only known pre-Tannakian categories in characteristic~0 were those attached to (super)groups, and those coming from Deligne interpolation (see \cite{Deligne3} and \cite{ComesOstrik1, ComesOstrik, EntovaAizenbudHeidersdorf, Harman1, Harman2, Knop, Knop2, Mori}).

In recent work \cite{repst}, we gave a general construction of pre-Tannakian categories. Recall that an \defn{oligomorphic group} is a permutation group $(G, \Omega)$ such that $G$ has finitely many orbits on $\Omega^n$ for all $n \ge 0$. Given such a group, and an additional piece of data $\mu$ called a measure, we construct a (non-abelian) tensor category $\uPerm(G; \mu)$. We also construct an abelian category $\uRep(G; \mu)$ that is well-behaved if the measure $\mu$ is ``normal,'' and sometimes pre-Tannakian. These constructions extend to a slightly larger class of groups, the \defn{admissible groups}, which are essentially inverse limits of oligomorphic groups. See \S \ref{s:olig} for a summary.

The infinite symmetric group is oligomorphic, and in this case our construction recovers Deligne's category $\uRep(\fS_t)$. A new example coming out of our theory is the Delannoy category: this is a pre-Tannakian category $\uRep(G; \mu)$ where $G$ is the group of orientation preserving self-homeomorphisms of the real line. The paper \cite{line} is a detailed study of the Delannoy category, and \cite{circle} treats a related example.

While oligomorphic groups carry an important topology, they are discrete in spirit. In \cite{homoten}, we defined and constructed algebraic analogs of oligomorphic groups, and we expect that there is an analogous story to \cite{repst} for them. In particular, Deligne's interpolation category $\uRep(\GL_t)$ should be attached to the algebraic-oligomorphic group $\GL_{\infty}$. However, we do not yet know how this works.

Our working hypothesis is that algebraic-oligomorphic groups explain all pre-Tannakian categories in characteristic~0; the evidence for this idea is minimal though, so we will not call it a conjecture. The purpose of this paper is to confirm this hypothesis in the cases that are presently accessible. We believe this is the first piece of progress on Problem~\ref{prob:main} outside of the ``essentially Tannakian'' case (under which we include the super-Tannakian and Verlinde--Tannakian cases).

\subsection{Discrete pre-Tannakian categories}

We now introduce the class of categories that feature in our main theorem. We begin with some observations about algebraic groups.

Let $G$ be an algebraic group over a field $k$. Suppose that $G$ is discrete, i.e., its identity component is trivial. Then $G$ acts faithfully on a finite set $X$ (e.g., itself). In algebraic geometry terms, $X$ is the spectrum of a finite dimensional \'etale $k$-algebra $A$. We can view $A$ as an algebra object in the category $\Rep(G)$; since the action of $G$ is faithful, it follows that $A$ generates $\Rep(G)$ as a tensor category. As sums of tensor powers of $A$ are also \'etale algebras, this shows that every object of $\Rep(G)$ is a subquotient of an \'etale algebra. This reasoning is reversible too: if $G$ is an algebraic group and every object of $\Rep(G)$ is a subquotient of an \'etale algebra then $G$ is discrete.

The notion of \'etale algebra makes sense in any tensor category (see \S \ref{s:etale}). Motivated by the above discussion, we make the following definitions:

\begin{definition}
A pre-Tannakian category $\cC$ is \defn{discrete}\footnote{A category is called \defn{discrete} if it is equivalent to a set (regarded as a category). Since a non-zero $k$-linear category is never discrete in this sense, we do not expect our usage of the term to cause any confusion.} (resp.\ \defn{strongly discrete}) if every object is a subquotient (resp.\ quotient) of an \'etale algebra.
\end{definition}

We make some comments on this definition. For algebraic groups, discrete and strongly discrete are equivalent. This is not true generally: for example, in positive characteristic, Deligne's interpolation category $\uRep(\fS_t)$ is discrete but not strongly discrete (Example~\ref{ex:not-strong}). We show that a discrete pre-Tannakian category with enough projectives is strongly discrete (Proposition~\ref{prop:strong-proj}); in particular, a semi-simple discrete pre-Tannakian category is strongly discrete. We also show that a discrete pre-Tannakian  category is strongly discrete if and only if exact sequences split \'etale locally (Proposition~\ref{prop:strong}).

Finally, we emphasize that discrete pre-Tannakian categories are a natural class of categories to consider: they abstract many of the important properties of representation categories of finite groups. Thus, in a way, they generalize the notion of finite group.

\subsection{The main theorem}

As stated, our working hypothesis is that pre-Tannakian categories in characteristic~0 come from algebraic-oligomorphic groups. This predicts that discrete pre-Tannakian categories come from ordinary oligomorphic groups, as these are the algebraic-oligomorphic groups with trivial identity component. Our main theorem confirms this (in any characteristic):

\begin{theorem} \label{mainthm}
Let $\cC$ be a discrete pre-Tannakian category over an algebraically closed field $k$. Then there exists an admissible group $G$ and a $k$-valued measure $\mu$ for $G$ such that $\cC$ is a weak abelian envelope of $\uPerm(G; \mu)$. If $\cC$ is strongly discrete then $\cC$ is the abelian envelope of $\uPerm(G; \mu)$, the measure $\mu$ is normal, and $\cC$ is equivalent to $\uRep^{\rf}(G; \mu)$.
\end{theorem}

If $\cC$ is finitely generated as a tensor category then one can take $G$ to be oligomorphic (Remark~\ref{rmk:olig}). We refer to \S \ref{ss:abenv} for the definition of (weak) abelian envelope. Abelian envelopes are unique up to equivalence, but weak abelian envelopes need not be unique in general (Example~\ref{ex:weak}). Thus the theorem does not necessarily pin down $\cC$ uniquely in the general case. However, in the strongly discrete case, it determines $\cC$ very precisely. We highlight one special case of the theorem:

\begin{corollary}
Let $\cC$ be a semi-simple discrete pre-Tannakian category over an algebraically closed field $k$. Then there exists $G$ and $\mu$ as in the theorem such that $\cC$ is equivalent to the Karoubian envelope of $\uPerm(G; \mu)$. Moreover, $\mu$ is a regular measure (see \S \ref{ss:meas}).
\end{corollary}

Theorem~\ref{mainthm} shows that there is a fundamental connection between pre-Tannakian categories and oligomorphic groups. This justifies our work in \cite{repst}: the categories constructed there are not merely curious examples, but constitute a natural class.

\subsection{The oligomorphic fundamental group}

Perhaps the most interesting aspect of Theorem~\ref{mainthm} is that it produces a group from the category $\cC$. The construction of this group closely parallels that of the \'etale fundamental group, and does not require any discreteness hypothesis.

We recall the basic construction of the \'etale fundamental group. Let $X$ be an irreducible algebraic variety over an algebraically closed field $k$. Let $\bS(X)$ be the category of finite \'etale covers of $X$; we note that this is opposite to the category of finite \'etale algebras in the category $\operatorname{Coh}(X)$. The category $\bS(X)$ is a Galois category, meaning it is equivalent to the category of discrete $G$-sets for some profinite group $G$; in fact, $G$ can be constructed as the automorphism group of a fiber functor on $\bS(X)$. The \'etale fundamental group of $X$ is the group $G$.

Now let $\cC$ be a pre-Tannakian category. Let $\bS(\cC)$ be the opposite of the category of \'etale algebras in $\cC$. A key result in this paper (Theorem~\ref{thm:frob}) is that $\bS(\cC)$ is a \defn{pre-Galois} category, a class of categories introduced in \cite{bcat}. Pre-Galois categories are to Galois categories what pre-Tannakian categories are to Tannakian categories; in other words, they satisfy similar abstract categorical properties, but do not need to admit a fiber functor. The main theorem of \cite{bcat} states that any pre-Galois category is equivalent to the category $\bS(G)$ of (finitary smooth) $G$-sets, for some admissible group $G$. This enables us to make the following definition:

\begin{definition}
The \defn{oligomorphic fundamental group} of a pre-Tannakian category $\cC$, denoted $\pi^{\rm olig}(G)$, is an admissible group $G$ such that $\bS(\cC)$ is equivalent to $\bS(G)$.
\end{definition}

The group $G$ is not exactly well-defined, but this is typically not an issue; see \S \ref{ss:pre-gal} for details. Hypothetically, if $\cC$ is associated to an algebraic-oligomorphic group $H$, then we expect $\pi^{\rm olig}(\cC)$ to recover the component group $\pi_0(H)$ of $H$. We note that the oligomorphic fundamental group is similar in spirit to Deligne's fundamental group $\pi(\cC)$ \cite[\S 8]{Deligne1}; however, Deligne's group is a group object internal to the category of affine schemes in $\cC$, while the oligomorphic fundamental group is an actual group.

We show in Theorem~\ref{thm:meas} that $\pi^{\rm olig}(\cC)$ carries a natural measure $\mu$, and that the category $\uPerm(\pi^{\rm olig}(\cC), \mu)$ naturally maps to $\cC$. With this result in hand, Theorem~\ref{mainthm} follows without too much difficulty.

\subsection{Remaining problems}

While Theorem~\ref{mainthm} is a significant step towards a classification of discrete pre-Tannakian categories, more work is needed to obtain a truly complete classification. The following seem to be the three most important problems.
\begin{enumerate}
\item In Theorem~\ref{mainthm}, is $\cC$ the abelian envelope of $\uPerm(G; \mu)$ even when $\cC$ is not strongly discrete? If not, is the weak abelian envelope of $\uPerm(G; \mu)$ at least unique?
\item Given an oligomorphic group $G$ and a measure $\mu$, does a (weak) abelian envelope of $\uPerm(G; \mu)$ exist? We proved this under rather stringent hypotheses in \cite{repst}.
\item Finally, we would like more information about measures on oligomorphic groups. For instance, can one give general conditions on a group that ensures a measure exists?
\end{enumerate}

\subsection{Notation}

We list some of the important notation here:
\begin{description}[align=right,labelwidth=2.5cm,leftmargin=!]
\item[ $k$ ] the coefficient field
\item[ $\bzero$ ] the initial object of a category (e.g., the empty set)
\item[ $\bone$ ] the final object of a category (e.g., the one-point set)
\item[ $\bbone$ ] unit object of a tensor category (e.g., the trivial representation)
\item[ $\bS(G)$ ] the category of finitary $G$-sets
\item[ $\Et(\cC)$ ] the category of \'etale algebras in $\cC$
\item[ $\bS(\cC)$ ] the opposite category to $\Et(\cC)$
\end{description}

\section{Tensor categories} \label{s:tensor}

\subsection{Basic definitions}

We now recall some definitions related to tensor categories. We fix a field $k$ for the duration of the paper.

\begin{definition}
A \defn{tensor category} is a $k$-linear additive category equipped with a symmetric monoidal structure $\otimes$ that is $k$-bilinear. We write $\bbone$ for the unit object. A \defn{tensor functor} is a $k$-linear symmetric monoidal functor. 
\end{definition}

\begin{definition}
Let $X$ be an object in a tensor category $\cC$. A \defn{dual} of $X$ is an object $X^{\vee}$ equipped with maps
\begin{displaymath}
\alpha \colon \bbone \to X \otimes X^{\vee}, \qquad \beta \colon X^{\vee} \otimes X \to \bbone,
\end{displaymath}
called \defn{co-evaluation} and \defn{evaluation}, such that the compositions
\begin{displaymath}
\xymatrix@C=3em{
X \ar[r]^-{\alpha \otimes \id} & X \otimes X^{\vee} \otimes X \ar[r]^-{\id \otimes \beta} & X }
\end{displaymath}
\vskip-4ex
\begin{displaymath}
\xymatrix@C=3em{
X^{\vee} \ar[r]^-{\id \otimes \alpha} & X^{\vee} \otimes X \otimes X^{\vee} \ar[r]^-{\beta \otimes \id} & X^{\vee} }
\end{displaymath}
are the identity. The dual is unique up to unique isomorphism, if it exists. We say that $X$ is \defn{rigid} if it has a dual, and we we say that $\cC$ is \defn{rigid} if every object is.
\end{definition}

We note that if $X$ is rigid object in a tensor category then the functor $X \otimes -$ has both a left and a right adjoint (given by $X^{\vee} \otimes -$). In particular, the tensor product is exact in a rigid abelian tensor category. See \cite[\S 2.10]{EGNO} for details on duals.

\begin{definition} \label{defn:pt}
A tensor category $\cC$ is \defn{pre-Tannakian} if it satisfies the following conditions:
\begin{enumerate}
\item The category $\cC$ is abelian and every object has finite length.
\item For any two objects $X$ and $Y$, the $k$-vector space $\Hom(X,Y)$ is finite dimensional.
\item The category $\cC$ is rigid.
\item We have $\End(\bbone)=k$. \qedhere
\end{enumerate}
\end{definition}

We refer to the book \cite{EGNO} for general background on tensor categories. We warn the reader that \cite{EGNO} uses the term ``tensor category'' in a much stronger sense than us: ``pre-Tannakian'' in our sense is equivalent to ``symmetric tensor category'' in the sense of \cite{EGNO}. Our use of ``pre-Tannakian'' follows \cite[\S 2.1]{ComesOstrik}.

\subsection{Algebras}

Fix a tensor category $\cC$ for the remainder of \S \ref{s:tensor}. For our purposes, an \defn{algebra} in $\cC$, is an associative, unital, and commutative algebra object. We write $\Alg(\cC)$ for the category of algebras in $\cC$. For an algebra $A$, we write $\mu_A \colon A \otimes A \to A$ and $\eta_A \colon \bbone \to A$ for the multiplication and unit maps on $A$, and we drop the subscripts when possible. If $A$ and $B$ are algebras then $A \oplus B$ and $A \otimes B$ are naturally algebras. These constructions are the product and co-product in $\Alg(\cC)$, respectively. The algebras 0 and $\bbone$ are the final and initial objects of $\Alg(\cC)$, respectively. The following concept will feature prominently:

\begin{definition}
An algebra $A$ is \defn{atomic} if it is non-zero and does not decompose as a product, i.e., given an algebra isomorphism $A \cong B \oplus C$ we have $B=0$ or $C=0$.
\end{definition}

Let $A$ be an algebra. An \defn{$A$-module} is an object $M$ of $\cC$ equipped with a map $\mu_M \colon A \otimes M \to M$ satisfying the usual conditions. We note that $\bbone$-modules are exactly objects of $\cC$. We write $\Mod_A$ for the category of $A$-modules. An \defn{ideal} of $A$ is an $A$-submodule of $A$. Suppose $\cC$ is abelian. Given two $A$-modules $M$ and $N$, we define $M \otimes_A N$ as the co-equalizer of the natural maps $A \otimes M \otimes N \rightrightarrows M \otimes N$. With this structure, $\Mod_A$ is itself an abelian tensor category. We note that the unit object in $\Mod_A$ is $A$ itself. This observation allows us to reduce many statements about general algebras to the case of the unit object.

\subsection{Elemental notation}

Working with morphisms in tensor categories can be notationally cumbersome. We will therefore sometimes adopt a functor of points approach to allow us to use more familiar (and compact) notation. We illustrate this device with two examples here.

First, suppose that $A$ is an algebra in $\cC$. A \defn{$T$-point} of $A$ is simply a morphism $T \to A$ in $\cC$. If $x$ is a $T_1$-point of $A$ and $y$ is a $T_2$-point of $A$, we write $x \cdot y$ or $xy$ for the $(T_1 \otimes T_2)$-point given by the composition
\begin{displaymath}
\xymatrix@C=3em{
T_1 \otimes T_2 \ar[r]^-{x \otimes y} & A \otimes A \ar[r]^-{\mu} & A. }
\end{displaymath}
If $z$ is a $T_3$-point of $A$, then the associative law for $A$ takes the usual form
\begin{displaymath}
(x \cdot y) \cdot z = x \cdot (y \cdot z).
\end{displaymath}
Note that the two sides are morphisms $T_1 \otimes T_2 \otimes T_3 \to A$. Taking $x=y=z$ to be the identity map $A \to A$ recovers the standard definition of associativity.

Second, suppose that $\beta \colon X \otimes Y \to \bbone$ is a pairing. Given a $T_1$-point $x$ of $X$ and a $T_2$-point $y$ of $Y$, we write $\beta(x,y)$ for the composition
\begin{displaymath}
\xymatrix@C=3em{
T_1 \otimes T_2 \ar[r]^-{x \otimes y} & X \otimes Y \ar[r]^-{\beta} & \bbone. }
\end{displaymath}
Suppose that $\beta$ is a perfect pairing, i.e., there exists a morphism $\alpha \colon \bbone \to Y \otimes X$ such that $\alpha$ and $\beta$ are the co-evaluation and evaluation maps in a duality. We then have the following familiar statement: if $x$ is a $T_1$-point of $X$ such that $\beta(x,y)=0$ for all $T_2$-points $y$ of $Y$ (as $T_2$ varies) then $x=0$. Indeed, taking $y$ to be the identity of $Y$, we see that the composition
\begin{displaymath}
\xymatrix@C=3em{
T_1 \otimes Y \ar[r]^-{x \otimes \id} & X \otimes Y \ar[r]^-{\beta} & \bbone }
\end{displaymath}
is zero. It follows that the composition
\begin{displaymath}
\xymatrix@C=3em{
T_1 \ar[r]^-{\id \otimes \alpha} & T_1 \otimes Y \otimes X \ar[r]^-{x \otimes \id} & X \otimes Y \otimes X \ar[r]^-{\beta \otimes \id} & X }
\end{displaymath}
is zero as well. But this is $x$ by the axioms for $\alpha$ and $\beta$ (note that the order of the first two morphisms can essentially be switched).

\subsection{Invariant algebras}

For an object $M$ of $\cC$, we put $\Gamma(M)=\Hom_{\cC}(\bbone, M)$. This is a $k$-vector space that we refer to as the \defn{invariants} of $M$. If $A$ is an algebra in $\cC$ then $\Gamma(A)$ is an associative, unital, and commutative $k$-algebra, which we call the \defn{invariant algebra} of $A$. One easily sees that product decompositions of $A$ correspond to idempotents of $\Gamma(A)$. In particular, $A$ is atomic if and only if the only idempotents of $\Gamma(A)$ are $0 \ne 1$.

Recall that the following conditions on a commutative ring $R$ are equivalent:
\begin{enumerate}
\item $R$ is \defn{absolutely flat}, i.e., every $R$-module is flat.
\item Every element of $R$ is a product of a unit and an idempotent.
\item $R$ is von Neumann regular, i.e., for every $x \in R$ there is a $y \in R$ such that $x^2y=x$ and $y^2x=y$; the element $y$ is called the \defn{weak inverse} of $x$, and is unique.
\item $R$ is a subring of a product of fields closed under taking weak inverses.
\end{enumerate}

\begin{proposition} \label{prop:rigid-unit}
Suppose that $\cC$ is rigid abelian. Then $\Gamma(\bbone)$ is absolutely flat. Moreover, the map
\begin{displaymath}
\{ \text{idempotents of $\Gamma(\bbone)$} \} \to \{ \text{ideals of $\bbone$} \}
\end{displaymath}
taking an idempotent to the principal ideal it generates is a bijection.
\end{proposition}

\begin{proof}
The statement about ideals is \cite[Remark~1.18]{DeligneMilne}. Absolute flatness of $\Gamma(\bbone)$ follows from this, as we now explain. Let $a \in \Gamma(\bbone)$. The isomorphism $\beta \colon \bbone \otimes \bbone \to \bbone$ is a self-duality. Let $x$ and $y$ be points of $A$. Then $\beta(ax,y)=axy=\beta(x,ay)$. It follows that $ax=0$ is equivalent to $0=\beta(ax,y)=\beta(x,ay)$ for all $y$. We thus see that $\ker(a)$ is exactly the orthogonal complement of $\im(a)$, and so we have $A=\ker(a) \oplus \im(a)$ by \cite[Proposition~1.17]{DeligneMilne}. It follows that $a \colon \im(a) \to \im(a)$ is an isomorphism. Let $b$ be its inverse, extended by~0 to $\ker(a)$. Then $a$ and $b$ are weak inverses of each other, and so $\Gamma(\bbone)$ is von Neumann regular.
\end{proof}

The proposition allows one to understand how $\bbone$ decomposes in various situations. Here is one important example.

\begin{corollary} \label{cor:pretan-unit}
If $\cC$ is pre-Tannakian then $\bbone$ is simple.
\end{corollary}

\begin{proof}
Since $\Gamma(\bbone)$ is a field, it has no non-trivial idempotents, and so $\bbone$ has no non-trivial ideals, which exactly means it is a simple object.
\end{proof}

We close this discussion with one additional property of $\Gamma$ in the pre-Tannakian case.

\begin{proposition} \label{prop:Gamma-tensor}
Suppose $\cC$ is pre-Tannakian. Then for any objects $X$ and $Y$ in $\cC$, the natural map $\Gamma(X) \otimes_k \Gamma(Y) \to \Gamma(X \otimes Y)$ is injective.
\end{proposition}

\begin{proof}
Put $V=\Gamma(X)$. We have a natural map $V \to X$, where here $V$ really means $V \otimes \bbone$, i.e., a direct sum of $\dim_k(V)$ copies of $\bbone$. Since $\bbone$ is simple, this map is injective. Since $\otimes$ is exact, the map $V \otimes Y \to X \otimes Y$ is injective. Since $\Gamma$ is left exact and additive, we find that the map $V \otimes_k \Gamma(Y) \to \Gamma(X \otimes Y)$ is injective, and this completes the proof.
\end{proof}

\subsection{Faithful flatness}

Suppose that $\cC$ is abelian. We say that $M$ is \defn{flat} if $M \otimes -$ is exact, and \defn{faithfully flat} if it is flat and moreover $M \otimes N=0$ implies $N=0$. We note that if $M$ is faithfully flat then a sequence
\begin{displaymath}
\xymatrix{ X_1 \ar[r] & X_2 \ar[r] & X_3 }
\end{displaymath}
is exact if and only if the sequence
\begin{displaymath}
\xymatrix{ M \otimes X_1 \ar[r] & M \otimes X_2 \ar[r] & M \otimes X_3 }
\end{displaymath}
is exact. We say that an $A$-module is (faithfully) flat if it is so as an object of the tensor category $\Mod_A$.

\begin{proposition}
Suppose $\otimes$ is exact. Let $M$ be an object of $\cC$, and suppose there is an injection $f \colon \bbone \to M' \otimes M$ for some object $M'$. Then $M$ is faithfully flat.
\end{proposition}

\begin{proof}
Every object of $\cC$ is flat by assumption. Let $N$ be a non-zero object of $\cC$. Since $\otimes$ is exact, the map $f \otimes \id \colon N \to M' \otimes M \otimes N$ is injective. Thus $M' \otimes M \otimes N$ is non-zero, and so $M \otimes N$ is non-zero, and so $M$ is faithfully flat.
\end{proof}

\begin{corollary} \label{cor:ff-1}
Suppose $\otimes$ is exact and $A$ is an algebra for which the unit map $\bbone \to A$ is injective. Then $A$ is faithfully flat.
\end{corollary}

\begin{corollary} \label{cor:nonzero-tensor}
Suppose $\cC$ is pre-Tannakian. Then every non-zero object is faithfully flat.
\end{corollary}

\begin{proof}
Let $M$ be a non-zero object of $\cC$. Since $\bbone$ is simple, the co-evaluation map $\bbone \to M \otimes M^{\vee}$ is injective, and so $M$ is faithfully flat.
\end{proof}

\section{Oligomorphic groups and their representations} \label{s:olig}

\subsection{Oligomorphic groups}

An \defn{oligomorphic group} is a permutation group $(G, \Omega)$ such that $G$ has finitely many orbits on $\Omega^n$ for all $n \ge 0$. Some examples are: the infinite symmetric group; the infinite general linear group over a finite field; the group of order-preserving self-bijections of the real line; and the automorphism group of the Rado graph. We refer to Cameron's book \cite{Cameron} for general background on oligomorphic groups.

Fix an oligomorphic group $(G, \Omega)$. For a finite subset $A$ of $\Omega$, let $G(A)$ be the subgroup of $G$ fixing each element of $A$. The groups $G(A)$ form a neighborhood basis of the identity for a topology on $G$. This topology is Hausdorff, non-archimedean (open subgroups form a neighborhood basis of the identity), and Roelcke pre-compact (if $U$ and $V$ are open subgroups then $V \backslash G/U$ is finite) \cite[\S 2.2]{repst}. We define an \defn{admissible group} to be a topological group with these three properties. Thus every oligomorphic group gives rise to an admissible group. While the main examples of interest are oligomorphic, theoretically it is more convenient to work with admissible groups.

Let $G$ be an admissible group. We say that an action of $G$ on a set $X$ is \defn{smooth} if all stabilizers are open. We use the term ``$G$-set'' to mean ``set equipped with a $G$-action.'' We say that a $G$-set is \defn{finitary} if it has finitely many orbits. We write $\bS(G)$ for the category of finitary $G$-sets, with morphisms being $G$-equivariant maps. An important property of $\bS(G)$ is that it is closed under products and fiber products; see \cite[\S 2.3]{repst}.

\subsection{Pre-Galois categories} \label{ss:pre-gal}

Let $\cB$ be a category with finite co-products. Write $\bzero$ for the initial object, and we say an object is \defn{empty} if it is isomorphic to $\bzero$. We say that an object $X$ is an \defn{atom} if it does not decompose under co-product, that is, $X \cong Y \amalg Z$ implies either $Y$ or $Z$ is empty. The following definition, introduced in \cite{bcat}, gives a combinatorial analog of pre-Tannakian categories.

\begin{definition} \label{defn:bcat}
We say that an essentially small category $\cB$ is \defn{pre-Galois category} if the following conditions hold:
\begin{enumerate}
\item The category $\cB$ has finite co-products.
\item Every object is isomorphic to a finite co-product of atoms.
\item Given objects $X$, $Y$, and $Z$, with $X$ an atom, the natural map
\begin{displaymath}
\Hom(X, Y) \times \Hom(X, Z) \to \Hom(X, Y \amalg Z)
\end{displaymath}
is a bijection.
\item The category $\cB$ has fiber products and a final object $\bone$.
\item A monomorphism of atoms is an isomorphism.
\item If $X \to Z$ and $Y \to Z$ are maps of atoms then $X \times_Z Y$ is non-empty.
\item The final object $\bone$ is atomic.
\item Equivalence relations in $\cB$ are effective (see \cite[Definition~4.9]{bcat}).
\end{enumerate}
We say that $\cB$ is a \defn{$\rB$-category} if it satisfies (a)--(e).
\end{definition}

If $G$ is an admissible group then the category $\bS(G)$ is a pre-Galois category, and this is the motivating example. The main result of \cite{bcat} establishes the converse:

\begin{theorem} \label{thm:pregal}
If $\cB$ is a pre-Galois category then there exists a admissible group $G$ such that $\cB \cong \bS(G)$.
\end{theorem}

As explained in \cite[\S 1.4]{bcat}, the group $G$ is not unique. However, this tends not to be so important since the properties of $G$ we care about only depend on $\bS(G)$. If $\cB$ has countably many isomorphism classes (which is typically the case in applications) then there is a unique $G$ (up to isomorphism) that is first-countable and complete.

\subsection{Measures and integration} \label{ss:meas}

We now recall some background on measures, in the sense of \cite{repst}. Let $G$ be an admissible group. A \defn{$\hat{G}$-set} is a set equipped with a smooth action of some open subgroup of $G$; shrinking the open subgroup does not change the $\hat{G}$-set. See \cite[\S 2.5]{repst} for details.

\begin{definition} \label{defn:meas}
Let $G$ be an admissible group. A \defn{measure} for $G$ valued in a commutative ring $k$ is a rule that assigns to each finitary $\hat{G}$-set $X$ a quantity $\mu(X)$ in $k$ such that the following conditions hold (in which $X$ and $Y$ denote finitary $\hat{G}$-sets):
\begin{enumerate}
\item Isomorphism invariance: if $X \cong Y$ then $\mu(X)=\mu(Y)$.
\item Normalization: if $X=\bone$ is a one-point set then $\mu(X)=1$.
\item Additivity: $\mu(X \amalg Y)=\mu(X)+\mu(Y)$.
\item Conjugation invariance: $\mu(X^g)=\mu(X)$, where $X^g$ is the conjugate of $X$ by $g \in G$.
\item Multiplicativity: if $\pi \colon X \to Y$ is a map of transitive $U$-sets, for some open subgroup $U$, and $F=\pi^{-1}(y)$ is the fiber over some point, then $\mu(X)=\mu(F) \cdot \mu(Y)$. \qedhere
\end{enumerate}
\end{definition}

Suppose $\mu$ is a measure and $f \colon X \to Y$ is a map of $G$-sets, with $X$ finitary and $Y$ transitive. Then all fibers of $f$ are conjugate, and thus have the same measure. We define $\mu(f)$ to be this common value. One can recover $\mu$ from the $\mu(f)$'s (for all $f$), and exactly specify the relations these values must satisfy. In other words, one can define a measure as a rule $f \mapsto \mu(f)$ satisfying certain conditions; see \cite[\S 4.5]{repst} for details. In fact, this formulation of measure makes sense for an arbitrary $\rB$-category.

If $\mu$ is a measure for $G$ and $U$ is an open subgroup of $G$ then $\mu$ restricts to a measure for $U$. The measure $\mu$ is called \defn{regular} if $\mu(X)$ is a unit of $k$ for every transitive $G$-set $X$. The restriction of a regular measure to an open subgroup is still regular. The measure $\mu$ is called \defn{quasi-regular} if it restricts to a regular measure on some open subgroup.

Suppose $\mu$ is a $k$-valued measure for $G$. We then obtain a theory of integration, as follows. Let $X$ be a $G$-set. We say that a function $\phi \colon X \to k$ is \defn{Schwartz} if it is smooth (i.e., it is left-invariant under an open subgroup of $G$) and of finitary support. In this case, $\phi$ assumes finitely many values, say $a_1, \ldots, a_r$, and we define
\begin{displaymath}
\int_X \phi(x) dx = \sum_{i=1}^r a_i \mu(\phi^{-1}(a_i)).
\end{displaymath}
See \cite[\S 3.3]{repst} for details. Let $\cC(X)$ be the space of Schwartz functions. If $f \colon X \to Y$ is a map of $G$-sets then there is a push-forward map $f_* \colon \cC(X) \to \cC(Y)$ defined by integrating over the fibers. This has the expected properties; see \cite[\S 3.4]{repst}. The measure $\mu$ is called \defn{normal} if $f_*$ is surjective whenever $f$ is surjective. A quasi-regular measure is normal.

\subsection{Representation categories} \label{ss:perm}

Let $G$ be an admissible group with a $k$-valued measure $\mu$. Given finitary $G$-sets $X$ and $Y$, a \defn{$Y \times X$ matrix} is a Schwartz function $A \colon X \times Y \to k$. We let $\Mat_{Y,X}$ be the space of such matrices. If $A$ is a $Y \times X$ matrix and $B$ is a $Z \times Y$ matrix, we define $BA$ to be the $Z \times X$ matrix given by
\begin{displaymath}
(BA)(z,x) = \int_Y B(z,y) A(y,x) dy.
\end{displaymath}
Matrix multiplication has the expected properties, and one can develop more of linear algebra in this context; see \cite[\S 7]{repst}.

A $Y \times X$ matrix $A$ defines an integral operator $A \colon \cC(X) \to \cC(Y)$. Suppose $f \colon X \to Y$ is a smooth function. We let $A_f$ and $B_f$ be the $Y \times X$ and $X \times Y$ matrices given by the indicator function of the graph of $f$ (or its transpose). As operators on Schwartz space, we have $A_f=f_*$ and $B_f=f^*$.

We now define a $k$-linear tensor category $\uPerm(G; \mu)$, as follows:
\begin{itemize}
\item The objects are formal symbols $\Vec_X$, with $X$ a finitary $G$-set.
\item $\Hom(\Vec_X, \Vec_Y)=\Mat_{Y,X}^G$ is the space of $G$-invariant matrices.
\item Composition is given by matrix multiplication.
\item The tensor product $\uotimes$ is defined on object by $\Vec_X \uotimes \Vec_Y = \Vec_{X \times Y}$.
\end{itemize}
See \cite[\S 8]{repst} for details. The category $\uPerm(G; \mu)$ is rigid: in fact, every object is self-dual \cite[\S 8.4]{repst}. However, it is almost never abelian. One should think of $\uPerm(G; \mu)$ as a category of ``permutation modules'' for $G$.

We now recall the definition of the abelian category $\uRep(G; \mu)$. The details of this construction will not be important, except in the proof of Lemma~\ref{lem:strong-3}, so we keep this discussion brief. We assume $G$ is first-countable for simplicity. Define the \defn{completed group algebra} $A=A(G)$ to be the inverse limit of Schwartz spaces $\cC(G/U)$ over open subgroups $U$. This becomes a ring under convolution; see \cite[\S 10.3]{repst}. We say that an $A$-module is \defn{smooth} if for each element the action of $A$ factors through some $\cC(G/U)$. We define $\uRep(G; \mu)$ to be the full subcategory of $\Mod_A$ spanned by the smooth modules. This is always a Grothendieck abelian category.

Suppose that $\mu$ is normal. Then we define a tensor structure on $\uRep(G; \mu)$ in \cite[\S 12]{repst}, and there is a natural fully faithful tensor functor $\uPerm(G; \mu) \to \uRep(G; \mu)$ \cite[Props.~11.12,~12.10]{repst}. If $\mu$ is quasi-regular and satisfies an additional technical condition called property~(P) then we show that every smooth $A$-module is the union of its finite length submodules, and that the finite length category $\uRep^{\rf}(G; \mu)$ is pre-Tannakian \cite[Theorem~13.2]{repst}.

\begin{remark} \label{rmk:rep-Ghat}
Suppose $\mu$ is normal and we know that $\uRep^{\rf}(G; \mu)$ is pre-Tannakian. Define $\uRep^{\rf}(\hat{G}; \mu)$ to be the 2-colimit of the categories $\uRep^{\rf}(U; \mu)$ over open subgroups $U$ of $G$. This is an interesting example of a rigid abelian tensor category that is not pre-Tannakian: typically, the $\Hom$ spaces are not finite dimensional, and objects are not of finite length.
\end{remark}

\section{\'Etale algebras in general tensor categories} \label{s:etale}

\subsection{Trace}

Fix a tensor category $\cC$ throughout \S \ref{ss:etale}. A \defn{rigid algebra} is an algebra whose underlying object is rigid. Let $A$ be a rigid algebra. We define the \defn{trace form} $\epsilon \colon A \to \bbone$ to be the composition
\begin{displaymath}
\xymatrix@C=3em{
A \ar[r]^-{\id \otimes \alpha_0} & A \otimes A \otimes A^{\vee} \ar[r]^-{\mu \otimes \id} & A \otimes A^{\vee} \ar[r]^-{\beta_0} & \bbone }
\end{displaymath}
where $\alpha_0$ and $\beta_0$ are (co-)evaluation maps. One can compute $\epsilon$ with respect to any dual of $A$, since duals are unique up to unique isomorphism. We define the \defn{trace pairing} $\beta \colon A \otimes A \to \bbone$ by $\beta = \epsilon \circ \mu$. We note that the trace pairing satisfies $\beta(xy,z)=\beta(x,yz)$.

\begin{proposition} \label{prop:trace-sum}
Let $A$ and $B$ be rigid algebras.
\begin{enumerate}
\item The algebras $A \oplus B$ and $A \otimes B$ are rigid.
\item We have $\epsilon_{A \oplus B}=\epsilon_A \oplus \epsilon_B$ and $\epsilon_{A \otimes B}=\epsilon_A \otimes \epsilon_B$.
\item The pairing $\beta_{A \oplus B}$ is the orthogonal direct sum of $\beta_A$ and $\beta_B$.
\item The pairing $\beta_{A \otimes B}$ is the tensor product of the pairings $\beta_A$ and $\beta_B$.
\end{enumerate}
\end{proposition}

\begin{proof}
The first statement follows from the general fact that a sum or tensor product of rigid objects is rigid. The remaining statements follow easily from the definitions.
\end{proof}

\subsection{\'Etale algebras} \label{ss:etale}

An algebra $A$ in $\cC$ is \defn{\'etale}\footnote{We only intend this definition to be used in ``finite'' situations, and it then gives a reasonable notion of ``finite \'etale.'' If $\cC$ is the category of all modules over a commutative ring $R$, our definition does not coincide with the usual definition of \'etale $R$-algebra. For example, $\bZ[1/2]$ is not an \'etale $\bZ$-algebra in our sense.} if it is rigid and its trace pairing $\beta$ is perfect. We write $\Et(\cC)$ for the full subcategory of $\Alg(\cC)$ spanned by \'etale algebras. By Proposition~\ref{prop:trace-sum}, the sum or tensor product of \'etale algebras is \'etale. These operations define the product and co-product on $\Et(\cC)$. The algebras 0 and $\bbone$ are \'etale, and are the final and initial objects of $\Et(\cC)$.

An \'etale algebra is identified with its own dual via the trace pairing. If $A$ and $B$ are \'etale algebras and $f \colon A \to B$ is any morphism in $\cC$ (not necessarily an algebra homomorphism), we let $f^{\vee} \colon B \to A$ be the dual of $f$ taken with respect to the self-dualities of $A$ and $B$.

\begin{proposition} \label{prop:eta-dual}
For an \'etale algebra $A$, we have $\epsilon=\eta^{\vee}$.
\end{proposition}

\begin{proof}
We have
\begin{displaymath}
\beta_A(\eta_A(x) \otimes y) = \eta_A(x) \cdot \epsilon(y) = \beta_{\bbone}(x \otimes \epsilon_A(y)),
\end{displaymath}
which shows that $\eta_A$ and $\epsilon_A$ are adjoint.
\end{proof}

The next proposition shows that \'etale algebras are stable under base change, in a general sense.

\begin{proposition} \label{prop:etale-functor}
Let $\Phi \colon \cC \to \cC'$ be a tensor functor. If $A$ is an \'etale algebra of $\cC$ then $\Phi(A)$ is an \'etale algebra of $\cC'$. Thus $\Phi$ induces a functor $\Et(\cC) \to \Et(\cC')$.
\end{proposition}

\begin{proof}
Since tensor functors preserve rigid objects, $\Phi(A)$ is a rigid algebra. It follows directly from the definitions that $\Phi(\epsilon_A)=\epsilon_{\Phi(A)}$. Therefore, $\Phi(\beta_A)=\beta_{\Phi(A)}$. Since $\beta_A$ is a perfect pairing and tensor functors preserve perfect pairings, it follows that $\beta_{\Phi(A)}$ is perfect. Thus $\Phi(A)$ is \'etale.
\end{proof}

\begin{remark}
We note that \'etale algebras in tensor categories have been previously studied; see, e.g., \cite{Davydov, DMNO, FFRS, LaugwitzWalton}.
\end{remark}

\subsection{Splittings}

Let $A$ be an algebra in $\cC$. A \defn{$\mu$-splitting} of $A$ is an $(A \otimes A)$-linear map $\delta \colon A \to A \otimes A$ such that $\mu \circ \delta = \id_A$. Here we regard $A$ as an $(A \otimes A)$-module via $\mu$. We say that $A$ is \defn{$\mu$-split} if it admits a $\mu$-splitting\footnote{Such algebras are sometimes called \defn{separable}.}. A \defn{splitting idempotent} of $A$ is an idempotent $\alpha \in \Gamma(A \otimes A)$ that satisfies $\mu(\alpha)=1$ and $(x \otimes 1)\alpha=(1 \otimes x)\alpha$ for any point $x$ of $A$. These two notions are equivalent:

\begin{proposition}
Let $A$ be an algebra in $\cC$. If $\delta$ is a $\mu$-splitting then $\delta(1)$ is a splitting idempotent. If $\alpha$ is a splitting idempotent then the map $\delta$ defined by $\delta(x)=(x \otimes 1) \alpha$ is a $\mu$-splitting. These two constructions are mutually inverse.
\end{proposition}

\begin{proof}
Let $\delta$ be a splitting idempotent, and put $\alpha=\delta(1)$. Since $\mu \circ \delta=\id_A$, we have $\mu(\alpha)=1$. We have
\begin{displaymath}
\delta(x) \cdot \delta(y) = \delta(\mu(\delta(x)) \cdot y) = \delta(xy),
\end{displaymath}
where in the first step we used that $\delta$ is $(A \otimes A)$-linear, and in the second that $\mu \circ \delta=\id_A$. Taking $x=y=1$, we find $\alpha^2=\alpha$, and so $\alpha$ is an idempotent. We have
\begin{displaymath}
(x \otimes 1) \alpha = (x \otimes 1) \delta(1) = \delta(x) = (1 \otimes x) \delta(1) = (1 \otimes x) \alpha.
\end{displaymath}
Thus $\alpha$ is a splitting idempotent. Also, $\delta$ can be recovered from $\alpha$ by the construction in the proposition statement.

Now suppose that $\alpha$ is a splitting idempotent, and define $\delta$ by $\delta(x)=(x \otimes 1) \alpha$. We have
\begin{displaymath}
(x \otimes y)\delta(z) = (x \otimes y)(z \otimes 1) \alpha = (xyz \otimes 1 \alpha) = \delta(xyz),
\end{displaymath}
where in the second step we used $(1 \otimes y)\alpha=(y \otimes 1) \alpha$. This shows that $\delta$ is $(A \otimes A)$-linear. We have $\mu(\delta(x))=x \mu(\alpha)=1$, and so $\delta$ is a splitting idempotent. It is clear that $\alpha=\delta(1)$, and so $\alpha$ can be recovered from $\delta$ by the construction in the proposition statement.
\end{proof}

The following proposition shows the relevance of these concepts to \'etale algebras.

\begin{proposition} \label{prop:mu-split-etale}
A rigid algebra is \'etale if and only if it is $\mu$-split. In this case, the $\mu$-splitting $\delta$ is dual to $\mu$ (and therefore unique), and satisfies $(\epsilon \otimes 1) \delta=\id$.
\end{proposition}

\begin{proof}
Let $A$ be a rigid $\mu$-split algebra. Let $\delta$ be a splitting and $\alpha=\delta(1)$ the associated idempotent. Consider $(\epsilon \otimes \id) \alpha$. This is equal to the composition
\begin{displaymath}
\xymatrix@C=3em{
\bbone \ar[r]^-{\alpha} & A \otimes A \ar[r]^-{\id \otimes \alpha_0} & A \otimes A \otimes A \otimes A^{\vee} \ar[r]^-{\id \otimes \mu \otimes \id} \ar[r] & A \otimes A \otimes A^{\vee} \ar[r]^-{\id \otimes \beta_0} & A }
\end{displaymath}
where $\alpha_0$ and $\beta_0$ are the (co-)evaluation maps. Since $\alpha$ is a splitting idempotent, we can change the third map to $\mu \otimes \id \otimes \id$. Since $\mu(\alpha)=1$, it follows that the above composition is also 1, i.e., $(\epsilon \otimes \id) \alpha=\eta$. We now have
\begin{displaymath}
(\epsilon \otimes \id) \delta(x) = (\epsilon \otimes \id) (1 \otimes x) \alpha = x (\epsilon \otimes \id) \alpha = x,
\end{displaymath}
that is, $(\epsilon \otimes \id) \delta = \id$. Therefore,
\begin{displaymath}
(\beta \otimes \id)(\id \otimes \alpha)(x) = (\epsilon \otimes 1)(\delta(x)) = x,
\end{displaymath}
which shows that $\beta$ and $\alpha$ define a self-duality of $A$, i.e., $A$ is \'etale. We have
\begin{displaymath}
\beta_{A \otimes A}(x \otimes y, \delta(z)) = (\epsilon \otimes \epsilon)((x \otimes y) \cdot \delta(z)) = (\epsilon \otimes \epsilon)(\delta(xyz))=\epsilon(xyz)=\beta_A(\mu(x \otimes y), z).
\end{displaymath}
In the third step, we factored $\epsilon \otimes \epsilon$ as $(\epsilon \otimes \id)(\id \otimes \epsilon)$ and used that $(\id \otimes \epsilon) \delta=\id$. This shows that $\delta$ and $\mu$ are adjoint, and so $\delta=\mu^{\vee}$. This shows that $\delta$ is unique.

Suppose now that $A$ is an \'etale algebra, and let $\delta=\mu^{\vee}$. Then
\begin{align*}
\beta_{A \otimes A}(x_1 \otimes x_2, (y_1 \otimes y_2) \delta(z)) &=
\beta_{A \otimes A}(x_1y_1 \otimes x_2y_2, \delta(z))
= \beta_A(x_1y_1x_2y_2, z) \\
&= \beta_A(x_1x_2, y_1y_2z)
= \beta_{A \otimes A}(x_1 \otimes x_2, \delta(y_1y_2z)),
\end{align*}
and so $\delta(y_1y_2z)=(y_1 \otimes y_2) \delta(z)$, i.e., $\delta$ is $(A \otimes A)$-linear. Let $\alpha \colon \bbone \to A \otimes A$ be the dual of $\beta$; this is the co-evaluation map in the self-duality of $A$. By definition, $\epsilon$ is the composition
\begin{displaymath}
\xymatrix@C=3em{
A \ar[r]^-{\id \otimes \alpha} & A \otimes A \otimes A \ar[r]^-{\mu \otimes \id} & A \otimes A \ar[r]^-{\beta} & \bbone. }
\end{displaymath}
Since $\beta=\epsilon \circ \mu$ and $\mu$ is associative, we can change $\mu \otimes \id$ to $\id \otimes \mu$ in the second map above. We thus find $\epsilon(x)=\beta(x,\mu(\alpha))$. Since we also have $\epsilon(x)=\beta(x,1)$, it follows that $\mu(\alpha)=1$, i.e., $\mu \circ \alpha=\eta$. Dualizing the equation $\beta=\epsilon \circ \mu$, we find $\alpha=\delta \circ \eta$, i.e., $\alpha=\delta(1)$, and so $\mu(\delta(1))=1$. Hence
\begin{displaymath}
\mu(\delta(x))=\mu((x \otimes 1) \cdot \delta(1))=x \cdot \mu(\delta(1))=x.
\end{displaymath}
Thus $\delta$ is a $\mu$-splitting.
\end{proof}

\subsection{Frobenius algebras}

We have seen that an \'etale algebra has various pieces of extra structure. The following definition helps organize all of this structure.

\begin{definition} \label{defn:frob}
A \defn{Frobenius algebra}\footnote{What we have called a ``Frobenius algebra'' is commonly called a ``special commutative Frobenius algebra.'' We have no need for more general Frobenius algebras.} in $\cC$ is an object $A$ equipped with maps
\begin{displaymath}
\eta \colon \bbone \to A, \quad \mu \colon A \otimes A \to A, \quad \epsilon \colon A \to \bbone, \quad \delta \colon A \to A \otimes A
\end{displaymath}
such that the following conditions hold:
\begin{enumerate}
\item $(A, \mu, \eta)$ is a commutative, associative, unital algebra object.
\item $(A, \delta, \epsilon)$ is a co-commutative, co-associative, co-unital co-algebra object.
\item We have
\begin{displaymath}
(\id \otimes \mu) \circ (\delta \otimes \id) = \delta \circ \mu = (\mu \otimes \id) \circ (\id \otimes \delta)
\end{displaymath}
as maps $A \otimes A \to A \otimes A$.
\item We have $\mu \circ \delta=\id_A$. \qedhere
\end{enumerate}
\end{definition}

The following is the prototypical example of a Frobenius algebra.

\begin{example} \label{ex:frob}
Let $X$ be a finite set, and let $A=k[X]$ be the vector space with basis $\{e_x\}_{x \in X}$. Then $A$ is a Frobenius algebra (in the category of vector spaces), with
\begin{displaymath}
\eta(1)=\sum_{x \in X} e_x, \qquad \mu(e_x \otimes e_y) = \delta_{x,y} e_x, \qquad \epsilon(e_x)=1, \qquad \delta(e_x)=e_x \otimes e_x,
\end{displaymath}
where the $\delta$ in the second equation is the Kronecker delta. The basis vectors are orthonormal with respect to the trace pairing $\beta$.
\end{example}

\begin{example} \label{ex:Vec-frob}
Let $G$ be an admissible group with a $k$-valued measure $\mu$. Then for any $G$-set $X$, the object $\Vec_X$ is naturally a Frobenius algebra in $\uPerm(G; \mu)$. The various maps are defined analogously to those in Example~\ref{ex:frob}.
\end{example}

\begin{example} \label{ex:Schwartz-frob}
Let $G$ and $\mu$ be as above, and suppose that $\mu$ is normal. For any finitary $G$-set $X$, the Schwartz space $\cC(X)$ is naturally a Frobenius algebra in the tensor category $\uRep(G; \mu)$. This follows from the previous example, since $\cC(X)$ is the image of $\Vec_X$ under the tensor functor $\uPerm(G; \mu) \to \uRep(G; \mu)$. However, it can also be seen directly: for instance, the multiplication on $\cC(X)$ is just the usual pointwise multiplication of functions on $X$.
\end{example}

Our next proposition connects Frobenius and \'etale algebras.

\begin{proposition} \label{prop:et-frob}
We have the following:
\begin{enumerate}
\item Let $A$ be a Frobenius algebra. Then $(A, \mu, \eta)$ is an \'etale algebra. Moreover, $\epsilon$ is the trace for this algebra, and $\delta$ is the $\mu$-splitting.
\item Let $(A, \mu, \eta)$ be an \'etale algebra. Let $\epsilon$ be the trace map on $A$, and let $\delta$ be the $\mu$-splitting for $A$. Then $(A, \mu, \eta, \delta, \epsilon)$ is a Frobenius algebra.
\end{enumerate}
\end{proposition}

\begin{proof}
(a) Let $A$ be a Frobenius algebra. Definition~\ref{defn:frob}(c,d) exactly say that $\delta$ is a $\mu$-splitting. Let $\alpha=\delta \circ \eta$ and $\beta=\epsilon \circ \mu$. Then
\begin{displaymath}
(\beta \otimes \id)(\id \otimes \alpha)(x)
= (\epsilon \otimes 1)((x \otimes 1) \cdot \delta(1))
= (\epsilon \otimes 1)(\delta(x)) = x.
\end{displaymath}
In the final step, we use that $\epsilon$ is a co-unit for $\delta$. We thus see that $A$ is self-dual. As $A$ is rigid and $\mu$-split, it is therefore \'etale (Proposition~\ref{prop:mu-split-etale}).

Computing with the above self-duality of $A$, we find that the trace form for $A$ is
\begin{displaymath}
\xymatrix@C=3em{
A \ar[r]^-{\id \otimes \eta} &
A \otimes A \ar[r]^-{\id \otimes \delta} \ar[r] &
A \otimes A \otimes A \ar[r]^-{\mu \otimes \id} &
A \otimes A \ar[r]^-{\mu} &
A \ar[r]^{\epsilon} & \bbone. }
\end{displaymath}
Since $\mu$ is associative, we can change the middle $\mu \otimes \id$ to $\id \otimes \mu$. Since $\mu \circ \delta=\id$ by Definition~\ref{defn:frob}(d), we can then cancel the second and third maps. Since $\mu \circ (\id \otimes \eta) = \id$, we can then cancel the first and fourth maps. We thus find that $\epsilon$ is the trace for $A$.

(b) Since $\epsilon$ and $\delta$ are dual to $\eta$ and $\mu$ (Propositions~\ref{prop:eta-dual} and~\ref{prop:mu-split-etale}), it follows that $(A, \delta, \epsilon)$ is a co-commutative, co-associative, co-unital co-algebra. We have already remarked that the conditions Definition~\ref{defn:frob}(c,d) are equivalent to $\delta$ being a $\mu$-splitting. We thus see that $A$ is a Frobenius algebra.
\end{proof}

Proposition~\ref{prop:et-frob}(b) shows that every \'etale algebra can be promoted to a Frobenius algebra; in fact, there is a unique such promotion, since Proposition~\ref{prop:et-frob}(a) specifies $\epsilon$ and $\delta$ in terms of the algebra structure. We thus see that \'etale and Frobenius algebras are equivalent concepts.

\section{\'Etale algebras in abelian tensor categories} \label{s:etale2}

\subsection{Modules}

Fix a rigid abelian tensor category $\cC$ throughout \S \ref{s:etale2}. Recall that if $A$ is an algebra in $\cC$ then $\Mod_A$ is itself a tensor category, with tensor product $\otimes_A$. In general, $\Mod_A$ will not be rigid; in fact, $\otimes_A$ is typically not even exact. However, we do get rigidity for \'etale algebras:

\begin{proposition} \label{prop:rigid-mod}
If $A$ is an \'etale algebra then $\Mod_A$ is a rigid tensor category.
\end{proposition}

\begin{proof}
Let $M$ be an $A$-module. Let $M^{\vee}$ be its dual in $\cC$, and let $\alpha_M$ and $\beta_M$ be the (co-)evaluation maps. The object $M^{\vee}$ admits a unique $A$-module structure such that $\beta_M(ax, y)=\beta_M(x,ay)$; we also have $(a \otimes 1) \alpha_M = (1 \otimes a) \alpha_M$ . In particular, we have $\beta_M(\alpha_A \cdot (x \otimes y))=\beta_M(x,y)$ and $\alpha_A \cdot \alpha_M=\alpha$, where $\alpha_A$ is the splitting idempotent of $A$ (recall $\mu_A(\alpha_A)=1$).

Define $\alpha_{M/A} \colon M \otimes_A M^{\vee} \to A$ to be the composition
\begin{displaymath}
\xymatrix@C=3em{
M \otimes_A M^{\vee} \ar[r]^-{\alpha_A} &
M \otimes M^{\vee} \ar[r]^-{1 \otimes \alpha_A} &
M \otimes M^{\vee} \otimes A \ar[r]^-{\beta_M \otimes \id} &
A }
\end{displaymath}
Let $\beta_{M/A}$ be the composition
\begin{displaymath}
\xymatrix@C=3em{
A \ar[r]^-{\id \otimes \alpha_M} & A \otimes M^{\vee} \otimes M \ar[r] & M^{\vee} \otimes_A M, }
\end{displaymath}
where the second map is the natural quotient map followed by scalar multiplication. One finds that
\begin{align*}
(\id \otimes \beta_{M/A})(\alpha_{M/A} \otimes \id)(x)
&= (1 \otimes \beta_M)(1 \otimes \alpha_A)(\alpha_A \otimes 1)(\alpha_M \otimes \id)(x) \\
&= (1 \otimes \beta_M)(\alpha_M \otimes \id)(x) = x,
\end{align*}
and similarly $(\beta_{M/A} \otimes \id)(\id \otimes \alpha_{M/A})=\id$. Thus $M^{\vee}$ is the dual of $M$ in $\Mod_A$.
\end{proof}

\begin{corollary} \label{cor:rigid-mod-1}
Let $A$ be an \'etale algebra.
\begin{enumerate}
\item The ring $\Gamma(A)$ is absolutely flat.
\item Ideals of $A$ correspond bijectively to idempotents of $\Gamma(A)$.
\item Any quotient algebra of $A$ is \'etale.
\end{enumerate}
\end{corollary}

\begin{proof}
As $A$ is the unit object in the rigid tensor category $\Mod_A$, the first two statements follow from Proposition~\ref{prop:rigid-unit}. As for (c), suppose that $I$ is an ideal of $A$. Then $I=\gamma A$ for some idempotent $\gamma$ of $\Gamma(A)$. Let $J=(1-\gamma) A$. Then $I$ and $J$ are algebras (with units $\gamma$ and $1-\gamma$) and $A=I \oplus J$ holds as algebras. Since $\cC$ is rigid, so are $I$ and $J$. Since $A$ is \'etale, it follows from Proposition~\ref{prop:trace-sum} that $I$ and $J$ are \'etale algebras. Finally, note that $A/I$ is isomorphic to $J$ as an algebra, and is thus \'etale.
\end{proof}

\begin{corollary} \label{cor:rigid-mod-2}
Suppose that $\cC$ is pre-Tannakian, and let $A$ be an \'etale algebra.
\begin{enumerate}
\item The invariant algebra $\Gamma(A)$ is an \'etale $k$-algebra, i.e., it is a finite product of finite separable field extensions of $k$.
\item $A$ decomposes into a finite product of atomic \'etale algebras.
\end{enumerate}
\end{corollary}

\begin{proof}
Since $A^{\otimes n}$ is \'etale, $\Gamma(A^{\otimes n})$ is reduced (Corollary~\ref{cor:rigid-mod-1}), and so $\Gamma(A)^{\otimes n}$ is reduced (by Proposition~\ref{prop:Gamma-tensor}). Thus $\Gamma(A)$ is a finite dimensional $k$-algebra such that every tensor power is reduced. It follows that $\Gamma(A)$ is \'etale. The minimal idempotents of $\Gamma(A)$ give a decomposition of $A$ into a finite product of algebras, which are \'etale by Proposition~\ref{prop:trace-sum}.
\end{proof}

Proposition~\ref{prop:rigid-mod} implies that the tensor product in $\Mod_A$ is exact when $A$ is \'etale. In fact, we have an even stronger result about this tensor product.

\begin{proposition}
Let $A$ be an \'etale algebra, and let $M$ and $N$ be $A$-modules. Then $M \otimes_A N$ is the summand of $M \otimes N$ corresponding to the projector $\alpha_A$.
\end{proposition}

\begin{proof}
We have
\begin{displaymath}
M \otimes_A N = (M \otimes N) \otimes_{A \otimes A} A,
\end{displaymath}
where here $A$ is regarded as an $(A \otimes A)$-module via $\mu$. Since $A$ is \'etale, $A$ is the summand of $A \otimes A$ corresponding to the idempotent $\alpha_A$, and so the result follows.
\end{proof}

In the pre-Tannakian case, we have the following upgrade of Proposition~\ref{prop:rigid-mod}.

\begin{proposition} \label{prop:pretan-mod}
Suppose that $\cC$ is pre-Tannakian and $A$ is an atomic \'etale algebra in $\cC$. Let $K=\Gamma(A)$, which is a finite separable extension of $k$. Then $\Mod_A$ is a pre-Tannakian category over $K$.
\end{proposition}

\begin{proof}
By Corollary~\ref{cor:rigid-mod-2}, $K$ is a finite separable extension of $k$. If $f \colon M \to N$ is a map of $A$-modules and $a \in K$ then we obtain a new map $af \colon M \to N$ by $a \otimes f \colon A \otimes_A M \to A \otimes_A M$. This shows that $\Mod_A$ is a $K$-linear category. Since the objects of $\Mod_A$ are objects of $\cC$ with extra structure, it is clear that all objects have finite length and all $\Hom$ spaces are finite dimensional. We have already seen that $\Mod_A$ is rigid (Proposition~\ref{prop:rigid-mod}). Finally, $\End_A(A)=K$. Thus $\Mod_A$ is pre-Tannakian.
\end{proof}

\begin{example} \label{ex:etale-1}
Suppose $G$ is a finite group and $\cC=\Rep(G)$. If $H$ is a subgroup then the space of functions on $G/H$ is easily seen to be an atomic \'etale algebra $A$ in $\cC$. The module category $\Mod_A$ is equivalent to $\Rep(H)$.
\end{example}

\begin{example} \label{ex:etale-2}
Suppose $G$ is an admissible group with a normal measure $\mu$ and $\cC=\uRep^{\rf}(G; \mu)$ is pre-Tannakian. Given an open subgroup $U$, the Schwartz space $A=\cC(G/U)$ is an \'etale algebra (Example~\ref{ex:Schwartz-frob}), and easily seen to be atomic. The module category $\Mod_A$ is equivalent to $\uRep^{\rf}(U; \mu)$.
\end{example}

\begin{example} \label{ex:etale-3}
Suppose $G$ and $\mu$ are as above, and let $\uRep^{\rf}(\hat{G}; \mu)$ be as in Remark~\ref{rmk:rep-Ghat}, which is a rigid abelian tensor category. If $X$ is a finitary $G$-set then $\cC(X)$ is an \'etale algebra in this category (this is similar to Example~\ref{ex:Schwartz-frob}). Its invariant algebra is again $\cC(X)$, but now as an algebra in the category of vector spaces. Corollary~\ref{cor:rigid-mod-1} thus implies that $\cC(X)$ is absolutely flat, which is also easy to see directly.
\end{example}

\subsection{\'Etale morphisms}

Let $f \colon A \to B$ be an algebra morphism in $\cC$. We say that $f$ is \defn{\'etale}, or that $B$ is an \defn{\'etale $A$-algebra}, if $B$ is an \'etale algebra in the category $\Mod_A$. Note that an algebra $B$ in $\cC$ is \'etale if and only if it is an \'etale $\bbone$-algebra.

\begin{proposition} \label{prop:etale-bc}
Let $A \to B$ be an \'etale algebra map, and let $A \to A'$ be an arbitrary algebra map. Then $A' \to A' \otimes_A B$ is \'etale.
\end{proposition}

\begin{proof}
The functor $\Mod_A \to \Mod_{A'}$ given by $A' \otimes_A -$ is a tensor functor, and therefore takes \'etale algebras to \'etale algebras (Proposition~\ref{prop:etale-functor}).
\end{proof}

\begin{proposition} \label{prop:et-comp}
Let $f \colon A \to B$ be an algebra homomorphism, with $A$ \'etale. Then $B$ is \'etale if and only if $f$ is \'etale.
\end{proposition}

\begin{proof}
Suppose that $f$ is \'etale. Let $\alpha_{B/A} \in \Gamma(B \otimes_A B)$ be a splitting idempotent for $B$ in $\Mod_A$, and let $\alpha_A \in \Gamma(A \otimes A)$ be the splitting idempotent for $A$. Put $\alpha_B=\alpha_A \alpha_{B/A} \in \Gamma(B \otimes B)$. (Here we are using the fact that multiplication by $\alpha_A$ identifies $B \otimes_A B$ with a summand of $B \otimes B$.) It is clear that $\alpha_B$ is idempotent and that $\mu_B(\alpha_B)=1$. If $x$ is a point of $B$ then $(x \otimes 1)\alpha_{B/A}=(1 \otimes x) \alpha_{B/A}$ holds in $B \otimes_A B$. Multiplying by $\alpha_A$, we see that $(x \otimes 1) \alpha_B = (1 \otimes x) \alpha_B$ holds in $B \otimes B$. Thus $\alpha_B$ is a splitting idempotent for $B$, and so $B$ is $\mu$-split. Since $\cC$ is rigid, it follows that $B$ is \'etale (Proposition~\ref{prop:mu-split-etale}).

Now suppose that $B$ is \'etale. Let $\alpha_{B/A}$ be the image of $\alpha_B$ in $B \otimes_A B$. This is clearly a splitting idempotent, and so $B$ is $\mu$-split in $\Mod_A$. Since $\Mod_A$ is rigid (Proposition~\ref{prop:rigid-mod}), it follows that $B$ is \'etale in $\Mod_A$ (Proposition~\ref{prop:mu-split-etale}).
\end{proof}

We note that the proposition implies that a composition of \'etale maps is \'etale.

\begin{corollary} \label{cor:et-fiber}
Let $A \to B$ and $A \to C$ be maps of \'etale algebras. Then $B \otimes_A C$ is an \'etale algebra. If $A$ is atomic, and $B$ and $C$ are non-zero then $B \otimes_A C$ is also non-zero.
\end{corollary}

\begin{proof}
The map $\bbone \to A$ is \'etale by assumption, the map $A \to B$ is \'etale (Proposition~\ref{prop:et-comp}), and the map $B \to B \otimes_A C$ is \'etale (Proposition~\ref{prop:etale-bc}). Thus $\bbone \to B \otimes_A C$ is \'etale (Proposition~\ref{prop:et-comp}). We now prove the final sentence. Since $A$ is atomic, its only ideals are~0 and $A$ (Corollary~\ref{cor:rigid-mod-1}). Since $A$ and $B$ are non-zero algebras, the map $A \to B$ is non-zero, and so it is necessarily injective. Thus $B$ is a faithfully flat $A$-module (Corollary~\ref{cor:ff-1}), and so $B \otimes_A C$ is non-zero since $C$ is non-zero.
\end{proof}

\subsection{Subalgebras} \label{ss:et-sub}

Let $B$ be an \'etale algebra in $\cC$. Our goal now is to classify the \'etale subalgebras of $B$. The classification will use the following concept. An \defn{E-idempotent} of $B$ is an idempotent $\gamma \in \Gamma(B \otimes B)$ satisfying the following conditions:
\begin{enumerate}
\item We have $\alpha \cdot \gamma=\alpha$, where $\alpha$ is the splitting idempotent for $B$.
\item $\gamma$ is symmetric, i.e., $\tau(\gamma)=\gamma$, where $\tau \colon B \otimes B \to B \otimes B$ is the symmetry isomorphism.
\item For $1 \le i<j \le 3$, let $\gamma_{i,j}$ be the idempotent of $B \otimes B \otimes B$ that puts $\gamma$ in the $i$ and $j$ factors and~1 in the remaining factor. Then $\gamma_{1,2} \gamma_{2,3} = \gamma_{1,2} \gamma_{1,3} = \gamma_{1,3} \gamma_{2,3}$.
\end{enumerate}
Given an $E$-idempotent $\gamma$, define $B^{\gamma}$ to be the kernel of the map
\begin{displaymath}
B \to B \otimes B, \qquad x \mapsto \gamma \cdot (x \otimes 1 - 1 \otimes x).
\end{displaymath}
If $\cC$ is the category of vector spaces then an E-idempotent of $B$ corresponds to an equivalence relation on $\Spec(B)$, and $\Spec(B^{\gamma})$ is the quotient space. Conversely, suppose $A$ is a subalgebra of $B$. Let $I(A)$ be the kernel of the map $B \otimes B \to B \otimes_A B$. By Corollary~\ref{cor:rigid-mod-1}, there is a unique idempotent $\gamma(A)$ of $\Gamma(B \otimes B)$ such that $I(A)=\ker(\gamma(B))$. The following is our main result:

\begin{proposition} \label{prop:etale-sub}
Let $B$ be an \'etale algebra. Then $A \mapsto \gamma(A)$ and $\gamma \mapsto B^{\gamma}$ define mutually inverse bijections
\begin{displaymath}
\{ \text{\'etale subalgebras of $B$} \} \leftrightarrow \{ \text{E-idempotents of $B$} \}.
\end{displaymath}
\end{proposition}

We require two lemmas before proving the proposition.

\begin{lemma} \label{lem:etale-sub-1}
Let $B$ be an \'etale algebra, let $\gamma$ be an E-idempotent of $B$, and put $A=B^{\gamma}$. Then:
\begin{enumerate}
\item $A$ is a subalgebra of $B$.
\item $\gamma$ belongs to $\Gamma(A \otimes A) \subset \Gamma(B \otimes B)$.
\item $A$ is \'etale, and $\gamma$ is a splitting idempotent for $A$.
\item $I(A)=\ker(\gamma)$.
\end{enumerate}
\end{lemma}

\begin{proof}
(a) It is clear that 1 belongs to $A$, i.e., $\eta$ factors through $A$. Suppose $x$ and $y$ are points of $A$. We have
\begin{displaymath}
xy \otimes 1 - 1 \otimes xy = (x \otimes 1)(y \otimes 1 - 1 \otimes y)+(x \otimes 1 - 1 \otimes x) (1 \otimes y).
\end{displaymath}
Since the right side is killed by $\gamma$, so is the left side. Thus $xy$ is a point of $A$, and so $A$ is closed under multiplication. Thus $A$ is a subalgebra of $B$.

(b) Since $\otimes$ is exact, $A \otimes A$ is the kernel of the map
\begin{align*}
B \otimes B &\to (B \otimes B \otimes B) \oplus (B \otimes B \otimes B) \\
x \otimes y &\mapsto \big( \gamma_{1,2} (x \otimes 1 - 1 \otimes x) \otimes y \big) \oplus \big( \gamma_{2,3} x \otimes (y \otimes 1 - 1 \otimes y) \big).
\end{align*}
The image of $\gamma$ under this map is
\begin{displaymath}
\gamma_{1,2} (\gamma_{1,3}-\gamma_{2,3}) \oplus \gamma_{2,3}(\gamma_{1,3}-\gamma_{1,2}),
\end{displaymath}
which vanishes by assumption. Thus $\gamma$ maps into $A \otimes A$.

(c) Let $\alpha$ be the splitting idempotent for $B$. Applying $\mu$ to the equation $\alpha \gamma=\alpha$, and using the fact that $\mu(\alpha)=1$, we find $\mu(\gamma)=1$. We have $(x \otimes 1) \gamma = (1 \otimes x) \gamma$ for any point $x$ of $A$ by definition of $A$. Thus $\gamma$ is a splitting idempotent of $A$, and so $A$ is \'etale (Proposition~\ref{prop:mu-split-etale}).

(d) We have $I(A) \subset \ker(\gamma)$ by definition of $A$. Since $\ker(\gamma)$ is generated by $1-\gamma$, it is enough to show that $1-\gamma \in I(A)$, or, equivalently, that $\gamma$ maps to~1 in $B \otimes_A B$. If $z$ is any element of $A \otimes A$ then its image in $B \otimes_A B$ is equal to $1 \otimes \mu_A(z)$. Since $\gamma$ belongs to $A \otimes A$, and is the splitting idempotent of $A$, we see that its image in $B \otimes_A B$ is $1 \otimes \mu_A(\gamma)=1$, as required.
\end{proof}

\begin{lemma} \label{lem:etale-sub-2}
Let $B$ be an \'etale algebra, and let $A$ be any subalgebra of $B$. Then $\gamma(A)$ is an E-idempotent of $B$.
\end{lemma}

\begin{proof}
Since $I(A)$ is symmetrical, it follows that $\gamma$ is. Let $\alpha$ be the splitting idempotent for $B$. Since $\alpha \cdot (x \otimes 1 - 1 \otimes x)=0$ and $\im(1-\gamma)$ is the ideal generated by $x \otimes 1 - 1 \otimes x$ with $x$ a point of $A$, we have $\alpha(1-\gamma)=0$, which shows $\alpha \le \gamma$. We now verify the third condition. We first claim that $\gamma_{1,2} \gamma_{2,3} \le \gamma_{1,3}$. It is equivalent to show $\ker(\gamma_{1,3}) \subset \ker(\gamma_{1,2} \gamma_{2,3})$. A typically point of $\ker(\gamma_{1,3})$ has the form
\begin{displaymath}
x \otimes 1 \otimes 1 - 1 \otimes 1 \otimes x.
\end{displaymath}
This is equal to
\begin{displaymath}
(x \otimes 1 \otimes 1 - 1 \otimes x \otimes 1) + (1 \otimes x \otimes 1 - 1 \otimes 1 \otimes x).
\end{displaymath}
The first term above belongs to $\ker(\gamma_{1,2})$, while the second belongs to $\ker(\gamma_{2,3})$. Since $\ker(\gamma_{1,2}\gamma_{2,3})$ both contains $\ker(\gamma_{1,2})$ and $\ker(\gamma_{2,3})$, the claim follows. Now, by symmetry, we have $\gamma_{i,j} \gamma_{j,k} \le \gamma_{i,k}$ whenever $\{i,j,k\}=\{1,2,3\}$. Multiplying by $\gamma_{i,j}$, we find $\gamma_{i,j} \gamma_{j,k} \le \gamma_{i,j} \gamma_{i,k}$. The reverse inequality holds by symmetry, and so the two sides are equal. We have thus shown that $\gamma$ is an E-idempotent.
\end{proof}

\begin{proof}[Proof of Proposition~\ref{prop:etale-sub}]
Let $\gamma$ be an E-idempotent of $B$, and put $A=B^{\gamma}$. Then $A$ is \'etale (Lemma~\ref{lem:etale-sub-1}). We now show that $\gamma=\gamma(A)$. For this, it is enough to show that the two idempotents have the same kernel. By definition, we have $\ker(\gamma(A))=I(A)$ by definition, and $\ker(\gamma)=I(A)$ by Lemma~\ref{lem:etale-sub-1}.

Now let $A$ be an \'etale subalgebra of $B$ and let $\gamma=\gamma(A)$. Then $\gamma$ is an E-idempotent of $B$ (Lemma~\ref{lem:etale-sub-2}). We claim that $A=B^{\gamma}$. This is equivalent to showing exactness of the sequence
\begin{displaymath}
\xymatrix{
A \ar[r] & B \ar[r] & B \otimes_A B, }
\end{displaymath}
where the second map is $x \otimes 1 - 1 \otimes x$; indeed, the kernel of this map is $B^{\gamma}$ by definition. This is simply flat descent; we recall the argument. First, $B$ is flat over $A$ since $\Mod_A$ is rigid (Proposition~\ref{prop:rigid-mod}), and faithfully flat since $A \to B$ is injective (Corollary~\ref{cor:ff-1}). Thus we can verify exactness of the above sequence after applying $B \otimes_A -$. This yields the sequence
\begin{displaymath}
\xymatrix{
B \ar[r]^-{d_1} & B \otimes_A B \ar[r]^-{d_2} & B \otimes_A B \otimes_A B, }
\end{displaymath}
where
\begin{displaymath}
d_1(x)=x \otimes 1, \qquad d_2(x \otimes y) = x \otimes y \otimes 1 - x \otimes 1 \otimes y.
\end{displaymath}
Define maps $s_1$ and $s_2$ in the opposite direction of $d_1$ and $d_2$ by
\begin{displaymath}
d_1(x \otimes y) = xy, \qquad d_2(x \otimes y \otimes z)=xy \otimes z.
\end{displaymath}
Then $d_1s_1-s_2d_2$ is the identity on $B \otimes_A B$, and so $\ker(d_2)=\im(d_1)$, as required.
\end{proof}

\begin{question}
Are all subalgebras of an \'etale algebra themselves \'etale?
\end{question}

\section{The oligomorphic fundamental group} \label{s:FrobB}

\subsection{The main theorem}

Let $\cC$ be a pre-Tannakian category. Recall that $\Et(\cC)$ is the category of \'etale algebras in $\cC$. We define $\bS(\cC)$ to be the opposite category to $\Et(\cC)$. Following Deligne \cite[\S 7.8]{Deligne1}, we think of $\bS(\cC)$ as the category of \'etale affine schemes in $\cC$. For an \'etale algebra $A$, we write $\Sp(A)$ for the corresponding object of $\bS(\cC)$. The following is the key technical theorem in this paper, and proved in \S \ref{ss:frobpf} below.

\begin{theorem} \label{thm:frob}
The category $\bS(\cC)$ is pre-Galois.
\end{theorem}

Combining this result with Theorem~\ref{thm:pregal}, we can now introduce the following invariant:

\begin{definition}
Let $\cC$ be a pre-Tannakian category. We define the \defn{oligomorphic fundamental group} of $\cC$, denoted $\pi^{\rm olig}(\cC)$, to be an admissible group $G$ such that $\bS(\cC) \cong \bS(G)$.
\end{definition}

We note that the group $G$ is not unique; see \S \ref{ss:pre-gal}. See \S \ref{ss:fundex} for some examples.

\begin{remark}
The category $\bS(\cC)$ can also be described as the category of Frobenius algebras in $\cC$, with morphisms being co-algebra homomorphisms.
\end{remark}

\subsection{Proof of Theorem~\ref{thm:frob}} \label{ss:frobpf}

We verify each condition in Definition~\ref{defn:bcat}.

(a) We have seen that $\Et(\cC)$ has finite products: if $A$ and $B$ are \'etale algebras then their product is the \'etale algebra $A \oplus B$. It follows that $\bS(\cC)$ has finite co-products, via
\begin{displaymath}
\Sp(A) \amalg \Sp(B) = \Sp(A \oplus B).
\end{displaymath}

(b) It is clear that $\Sp(A)$ is an atomic object of $\bS(\cC)$ if and only if $A$ is an atomic algebra. By Corollary~\ref{cor:rigid-mod-2}, every \'etale algebra is a finite product of atomic \'etale algebras. Thus every object of $\bS(\cC)$ is a finite co-product of atoms.

(c) Let $A$, $B$, and $C$ be \'etale algebras, with $A$ atomic. We show that the map
\begin{displaymath}
\Hom(\Sp(A), \Sp(B)) \amalg \Hom(\Sp(A), \Sp(C)) \to \Hom(\Sp(A), \Sp(B) \amalg \Sp(C))
\end{displaymath}
bijective. Equivalently, we must show that
\begin{displaymath}
\Hom(B, A) \amalg \Hom(C, A) \to \Hom(B \oplus C, A)
\end{displaymath}
is bijective, where the $\Hom$'s are taken in the category $\Et(\cC)$. It is clearly injective. Let $f \colon B \oplus C \to A$ be a given algebra map. Let $e \in \Gamma(B)$ and $e' \in \Gamma(C)$ be the unit elements. These are orthogonal idempotents summing to~1 in $\Gamma(B \oplus C)$. It follows that $f(e)$ and $f(e')$ are orthogonal idempotents of $\Gamma(C)$ summing to~1. Since $C$ is atomic, $\Gamma(C)$ is a field, and so $f(e)=0$ or $f(e')=0$; say the latter. Then $e'$ belongs to the kernel of $f$, and so all of $C$ is contained in the kernel of $f$. Hence $f$ is in the image of $\Hom(B,A)$ under the above map, which completes the proof.

(d) If $A \to B$ and $A \to C$ are maps of algebras in $\cC$ then $B \otimes_A C$ is the push-out in $\Alg(\cC)$. If $A$, $B$, and $C$ are \'etale then so is $B \otimes_A C$ (Corollary~\ref{cor:et-fiber}), and so it is the push-out in the category $\Et(\cC)$. Thus $\Et(\cC)$ has push-outs, and so $\bS(\cC)$ has fiber products. Explicitly,
\begin{displaymath}
\Sp(B) \times_{\Sp(A)} \Sp(C) = \Sp(B \otimes_A C).
\end{displaymath}

(e) Let $\Sp(f) \colon \Sp(B) \to \Sp(A)$ be a monomorphism of atomic objects of $\bS(\cC)$. We claim that $f$ is an isomorphism. Since $\Sp(f)$ is a monomorphism, the map
\begin{displaymath}
p_1 \colon \Sp(B) \times_{\Sp(A)} \Sp(B) \to \Sp(B)
\end{displaymath}
is an isomorphism, where $p_1$ is the projection onto the first factor. Translating to algebras, this says that the map
\begin{displaymath}
\id \otimes f \colon B \otimes_A A \to B \otimes_A B
\end{displaymath}
is an isomorphism. Since $A$ is atomic, $f \colon A \to B$ is injective (Corollary~\ref{cor:rigid-mod-1}), and so $B$ is faithfully flat over $A$ (Corollary~\ref{cor:ff-1}). Since $f$ becomes an isomorphism after applying $B \otimes_A -$, it follows that $f$ is an isomorphism.

(f) Consider maps $\Sp(B) \to \Sp(A)$ and $\Sp(C) \to \Sp(A)$, where $\Sp(B)$ and $\Sp(C)$ are non-empty and $\Sp(A)$ is atomic. Thus $B$ and $C$ are non-zero, and $A$ is atomic. We have
\begin{displaymath}
\Sp(B) \times_{\Sp(A)} \Sp(C) = \Sp(B \otimes_A C).
\end{displaymath}
Since $B \otimes_A C$ is non-zero (Corollary~\ref{cor:et-fiber}), it follows that the above fiber product is non-empty.

(g) We have already remarked that $\bbone$ is the initial object of $\Et(\cC)$, and so $\Sp(\bbone)$ is the final object of $\bS(\cC)$. Since $\cC$ is pre-Tannakian, $\bbone$ is simple (Corollary~\ref{cor:pretan-unit}), and thus atomic. Hence $\Sp(\bbone)$ is also atomic.

(h) Suppose that $R$ is an equivalence relation on $\Sp(B)$ in the category $\bS(\cC)$, i.e., $R$ is a subobject of $\Sp(B \otimes B)$ satisfying the equivalence relation axioms.

We first make a general observation. Let $\Sp(A)$ be an object of $\bS(\cC)$. A subobject of $\Sp(A)$ is a union of some of its atoms \cite[Corollary~3.14]{bcat}, and thus of the form $\Sp(\gamma \cdot A)$ for some idempotent $\gamma \in \Gamma(A)$. This establishes an order-preserving bijection between idempotents of $\Gamma(A)$ and subobjects of $\Sp(A)$, where here idempotents are ordered in the usual manner. In particular, if $X_1$ and $X_2$ are subobjects corresponding to idempotents $\gamma_1$ and $\gamma_2$ then $X_1 \cap X_2$ corresponds to $\gamma_1 \gamma_2$.

Now, let $\gamma$ be the idempotent of $\Gamma(B \otimes B)$ corresponding to $R$. We claim that $\gamma$ is an E-idempotent (see \S \ref{ss:et-sub}). The diagonal of $B \otimes B$ corresponds to the splitting idempotent $\alpha \in \Gamma(B \otimes B)$; since $R$ contains the diagonal, we thus have $\alpha \le \gamma$, i.e., $\alpha \gamma = \alpha$. Since $R$ is symmetrical, so is $\gamma$. Finally, in the obvious notation, we have $R_{1,2} \cap R_{2,3} = R_{1,3} \cap R_{2,3}=R_{1,2} \cap R_{1,3}$, which yields $\gamma_{1,2} \gamma_{2,3} = \gamma_{1,3} \gamma_{2,3} = \gamma_{1,2} \gamma_{1,3}$. This establishes the claim.

Let $A=B^{\gamma}$. We claim that $\Sp(B) \to \Sp(A)$ is the co-equalizer of the maps $R \rightrightarrows B$. Indeed, suppose $\Sp(B) \to \Sp(A')$ is some map that co-equalizes the two arrows. This exactly means that the two compositions
\begin{displaymath}
A' \to B \rightrightarrows \gamma \cdot (B \otimes B)
\end{displaymath}
are equal, where the two right maps send $x$ to $x \otimes 1$ and $1 \otimes x$. Thus $A'$ maps into the kernel of $x \mapsto \gamma (x \otimes 1 - 1 \otimes x)$, which is $A$ by definition. Since $A \to B$ is an injection, the factorization of $A'$ through $A$ is unique. This establishes the claim.

Finally, we claim that $R$ is the kernel pair of the map $\Sp(B) \to \Sp(A)$. This exactly means that $R=B \otimes_A B$, which follows from the relation $\gamma=\gamma(A)$ proved in Proposition~\ref{prop:etale-sub}. We thus see that the equivalence relation $R$ is effective.

\subsection{Examples} \label{ss:fundex}

We now give a number of examples of oligomorphic fundamental groups. We assume $k$ is algebraically closed of characteristic~0 in the following examples.

\textit{(a) Vector spaces.} Let $\cC=\Vec^{\rf}$ be the category of finite dimensional $k$-vector spaces. Then $\bS(\cC)$ is the category of finite \'etale schemes over $k$, which is equivalent to the category of finite sets. This is just $\bS(G)$ where $G$ is the trivial group, and so $\pi^{\rm olig}(\cC)=1$.

\textit{(b) Algebraic groups.} Let $G$ be an algebraic group over $k$, and let $\cC=\Rep(G)$ be the category of finite dimensional algebraic representations of $k$. Then $\bS(\cC)$ is the category of finite \'etale schemes over $k$ equipped with an algebraic action of $G$. Any such action factors through $\pi_0(G)$, and so $\bS(\cC)$ is equivalent to the category of finite $\pi_0(G)$-sets. Thus $\pi^{\rm olig}(\cC)=\pi_0(G)$.

\textit{(c) Deligne's $\fS_t$ category.} Let $\cC$ be the abelian version of Deligne's category $\uRep(\fS_t)$ \cite{Deligne3}, where $t \in k$, which is a pre-Tannakian category. This category is equivalent to $\uRep^{\rf}(\fS; \mu_t)$, where $\fS$ is the infinite symmetric group and $\mu_t$ is a particular measure on it. After possibly extending scalars, one can realize $\cC$ as an ultraproduct of ordinary representation categories $\Rep(\fS_n)$ (where both $n$ and the coefficient fields vary) \cite{Harman2}. From this perspective, an \'etale algebra in $\uRep(\fS_t)$ is the ultralimit of \'etale algebras in $\Rep(\fS_n)$'s. It is not difficult to classify these, and from this one finds that the $\cC(X)$'s are the only \'etale algebras in $\cC$; see \cite[Proposition~2.2]{Harman2} for details. Thus $\bS(\cC)$ is equivalent to the category $\bS(\fS)$ of finitary $\fS$-sets, and so $\pi^{\rm olig}(\cC)=\fS$. A similar approach works in other instances of Deligne interpolation.

\textit{(d) Deligne's $\GL_t$ category.} Deligne also defined an interpolation category $\cC=\uRep(\GL_t)$ for the general linear groups. We have $\pi^{\rm olig}(\cC)=1$ in this case, which can be seen using the ultraproduct approach from (c).

\textit{(e) The wreath product $\GL_u \wr \fS_t$.} One can construct a category $\uRep(\GL_u \wr \fS_t)$ associated to the ``wreath product'' of $\GL_u$ and $\fS_t$; see \cite{Mori}. The oligomorphic fundamental group is the infinite symmetric group $\fS$; this again follows from the ultraproduct perspective.

\textit{(f) Super vector spaces.} In a super-commutative algebra any odd element is nilpotent and therefore in the radical of the trace form. Therefore any \'etale algebra must be purely even, and we are back in the vector space case; thus $\pi^{\rm olig}(\mathrm{sVec})=1$.

\textit{(g) The Verlinde category.} Let $\mathrm{Ver}_p$ be the Verlinde category (see \cite[\S 3.2]{Ostrik}). This is a semi-simple pre-Tannakian category with $p-1$ simple objects $L_1, \ldots, L_{p-1}$, where $L_1=\bbone$. For $i \ne 1$, the simple $L_i$ is killed by a symmetric power (see \cite[Proposition~3.1]{Venkatesh}). Thus in a commutative algebra, any copy of $L_i$ with $i \ne 1$ belongs to the radical. Hence an \'etale algebra can only contain the simple $L_1$, and so $\pi^{\rm olig}(\mathrm{Ver}_p)=1$.

\subsection{The identity component} \label{ss:idcomp}

Suppose $G$ is an algebraic group. One can think of $G$ as being built of two pieces: its component group $\pi_0(G)$ and its identity component $G^{\circ}$. As we saw above, we can recover $\pi_0(G)$ from $\Rep(G)$ as the the oligomorphic fundamental group. It is therefore natural to ask if there is a general construction for pre-Tannakian categories that recovers $\Rep(G^{\circ})$ from $\Rep(G)$.

In fact, there is such a construction, though we have not been able to establish nice properties of it in general. Let $\cC$ be a pre-Tannakian category and let $H=\pi^{\rm olig}(\cC)$. For an open subgroup $U$ of $H$, define $\cC(U)$ to be the category of modules over the atomic \'etale algebra corresponding to $H/U$; this is pre-Tannakian by Proposition~\ref{prop:pretan-mod}. Now define $\hat{\cC}$ to be the 2-colimit of the categories $\cC(U)$ over the poset of open subgroups.

If $\cC=\Rep(G)$ for an algebraic group $G$ then $\hat{\cC}$ is indeed $\Rep(G^{\circ})$. Unfortunately, for a general pre-Tannakian category, $\hat{\cC}$ will not itself be pre-Tannakian (though it is always rigid); indeed, if $\cC=\uRep^{\rf}(G;\mu)$ then $\hat{\cC}=\uRep^{\rf}(\hat{G};\mu)$ (see Remark~\ref{rmk:rep-Ghat}). One might hope that the category of finite length objects in $\hat{\cC}$ forms a pre-Tannakian category. In this case, this category might reasonably be called the ``identity component'' of $\cC$. However, we have not been able to prove this.

\section{The measure on the oligomorphic fundamental group} \label{s:meas}

\subsection{Linearizations}

The goal of \S \ref{s:meas} is to define a measure on the oligomorphic fundamental group. To do this, we will appeal to the theory of linearizations in \cite{repst}. We now recall this theory. We fix a $\rB$-category $\cB$ in what follows, though we are really only interested in the case $\cB=\bS(G)$ for an admissible group $G$. We let $\cB^{\circ}$ be the category with the same objects as $\cB$ and where the morphisms are isomorphisms in $\cB$. We have a natural identification of $\cB^{\circ}$ with its opposite category.

A \defn{symmetric monoidal balanced functor} $\cB \to \cT$ is a pair $(\Psi, \Psi')$ of symmetric monoidal functors $\Psi \colon \cB \to \cT$ and $\Psi' \colon \cB^{\op} \to \cT$ such that $\Psi$ and $\Psi'$ have equal restriction to $\cB^{\circ}=(\cB^{\circ})^{\op}$; here $\cB$ carries the cartesian symmetric monoidal structure. We think of this as a rule that assigns to each object $X$ of $\cB$ an object $\Psi(X)$ of $\cT$, and to each morphism $f \colon X \to Y$ in $\cB$ morphisms $\alpha_f \colon \Psi(X) \to \Psi(Y)$ and $\beta_f \colon \Psi(Y) \to \Psi(X)$; here $\alpha_f=\Psi(f)$ and $\beta_f=\Psi'(f)$. We typically just write $\Psi$ in place of $(\Psi, \Psi')$.

Suppose we have a symmetric monoidal balanced functor $\Psi$. We introduce two conditions on $\Psi$. We say that $\Psi$ is \defn{additive} if
\begin{displaymath}
\Psi(X \amalg Y)=\Psi(X) \oplus \Psi(Y)
\end{displaymath}
in the natural manner, for all objects $X$ and $Y$ of $\cB$. To be precise, let $i \colon X \to X \amalg Y$ and $j \colon Y \to X \amalg Y$ be the natural maps. Then we demand
\begin{displaymath}
\beta_i \alpha_i = \id_{\Psi(X)}, \quad
\beta_i \alpha_j = 0, \quad
\beta_j \alpha_i = 0, \quad
\beta_j \alpha_j = \id_{\Psi(Y)}
\end{displaymath}
\begin{displaymath}
\alpha_i \beta_i + \alpha_j \beta_j = \id_{\Psi(X \amalg Y)}.
\end{displaymath}
For an object $X$ of $\cB$, we let $\alpha_X \colon \Psi(X) \to \bbone$ and $\beta_X \colon \bbone \to \Psi(X)$ be the maps induced by the natural map $X \to \bone$, where $\bone$ is the final object of $\cB$. We say that $\Psi$ is \defn{plenary} if for every atom $X$ of $\cB$ the space $\Hom(\Psi(X), \bbone)$ is one-dimensional and spanned by $\alpha_X$. We can now introduce the following important concept:

\begin{definition}
A \defn{linearization} of a $\rB$-category $\cB$ is a pair $(\cT, \Psi)$ where $\cT$ is a tensor category and $\Psi \colon \cB \to \cT$ is a symmetric monoidal balanced functor that is additive, plenary, and essentially surjective.
\end{definition}

If $\mu$ is a measure for $\cB$ then there is a natural linearization $\Psi_{\mu} \colon \cB \to \uPerm(\cB; \mu)$ given by $\Psi_{\mu}(X)=\Vec_X$ and $\alpha_f=A_f$ and $\beta_f=B_f$. The notion of measure in this context is discussed after Definition~\ref{defn:meas}, and the notation is as in \S \ref{ss:perm}. In fact, we proved \cite[Theorem~9.9]{repst} that these account for all linearizations:

\begin{theorem} \label{thm:linear}
Let $\Psi \colon \cB \to \cT$ be a linearization. Then there exists a $k$-valued measure $\mu$ for $\cB$ and an equivalence of tensor categories $\Phi \colon \uPerm(\cB; \mu) \to \cT$ such that $\Psi=\Phi \circ \Psi_{\mu}$ (actual equality). Both $\mu$ and $\Phi$ are unique.
\end{theorem}

The proof of Theorem~\ref{thm:linear} (specifically, the paragraph preceding \cite[Lemma~9.17]{repst}) shows exactly how to construct the measure $\mu$: given a morphism $f \colon X \to Y$ in $\cB$, with $Y$ an atom, $\mu(f)$ is the unique scalar such that $\alpha_f \beta_X = \mu(f) \beta_Y$.

\subsection{The main theorem}

We suppose our field $k$ is algebraically closed in what follows. Let $\cC$ be a pre-Tannakian category, and let $G=\pi^{\rm olig}(\cC)$ be its oligomorphic fundamental group. We identify $\bS(\cC)$ with $\bS(G)$ in what follows. For a finitary $G$-set, we let $k[X]$ be the corresponding \'etale algebra in $\cC$. For a morphism $f \colon X \to Y$ of finitary $G$-sets, we let $f^* \colon k[Y] \to k[X]$ be the corresponding algebra homomorphism, and we let $f_* \colon k[X] \to k[Y]$ be the dual of $f^*$.

Suppose $f \colon X \to Y$ is a morphism of $G$-sets, with $X$ finitary and $Y$ transitive. Then there is an induced map $f_* \colon \Gamma(k[X]) \to \Gamma(k[Y])$. Since $Y$ is atomic, $\Gamma(k[Y])=k$; indeed, $\Gamma(k[Y])$ is a finite extension field of $k$ by Proposition~\ref{prop:pretan-mod}, but $k$ is algebraically closed. We can thus regard $f_*(1)$ as an element of $k$. We put $\mu(f)=f_*(1)$.

\begin{theorem} \label{thm:meas}
The rule $\mu$ defined above is a measure on $G$. Furthermore, there is a natural fully faithful tensor functor
\begin{displaymath}
\Phi \colon \uPerm(G; \mu) \to \cC
\end{displaymath}
satisfying $\Phi(X)=k[X]$ and $\Phi(A_f)=f_*$ and $\Phi(B_f)=f^*$.
\end{theorem}

\begin{proof}
We define a symmetric monoidal balanced functor $\Psi \colon \bS(G) \to \cC$. For a finitary $G$-set $X$, we put $\Psi(X)=k[X]$. For a morphism $f \colon X \to Y$ of finitary $G$-sets, we define $\alpha_f=f_*$ and $\beta_f=f^*$. One readily verifies that $\Phi$ is naturally a symmetric monoidal balanced functor, and that $\Psi$ is additive. If $X$ is an atom then $\Hom_{\cC}(\Psi(X), \bbone)$ is dual to $\Gamma(k[X])$, and thus one-dimensional; since $\alpha_X$ is non-zero, it is a basis. We thus see that $\Psi$, regarded as a functor to its essential image, is a linearization of $\bS(G)$. The theorem now follows from Theorem~\ref{thm:linear}, and the remarks following it.
\end{proof}

\section{The main theorem} \label{s:main}

\subsection{Abelian envelopes} \label{ss:abenv}

In \S \ref{s:main}, we prove our main theorem. We first introduce some terminology surrounding abelian envelopes.

\begin{definition}
Let $\cP$ be a tensor category. A \defn{weak abelian envelope} of $\cP$ is a pair $(\cR, \Phi)$, where $\cR$ is a pre-Tannakian category and $\Phi \colon \cP \to \cR$ is a fully faithful tensor functor such that every object of $\cR$ is a subquotient of an object in the image of $\Phi$.
\end{definition}

\begin{definition} \label{defn:abenv}
Let $\cP$ be a tensor category. An \defn{abelian envelope} of $\cP$ is a pair $(\cR, \Phi)$, where $\cR$ is a pre-Tannakian category and $\Phi$ is a tensor functor, such that for any pre-Tannakian category $\cT$ the functor
\begin{displaymath}
\Fun^{\rm ex}_{\otimes}(\cR, \cT) \to \Fun^{\rm faith}_{\otimes}(\cP, \cT), \qquad \Psi \mapsto \Psi \circ \Phi
\end{displaymath}
is an equivalence, where $\Fun^{\rm ex}_{\otimes}$ denotes the category of exact tensor functors and $\Fun^{\rm faith}_{\otimes}$ the category of faithful tensor functors.
\end{definition}

Definition~\ref{defn:abenv} is taken from \cite[Definition~3.1.2]{CEAH}. An abelian envelope is unique (up to canonical equivalence) if it exists. An abelian envelope is also a weak abelian envelope. The following example shows that weak envelopes are not unique in general. However, for the categories $\uPerm(G; \mu)$, we know of no example where there are multiple weak envelopes.

\begin{example} \label{ex:weak}
Suppose $k$ has positive characteristic $p$, and let $\fS$ be the infinite symmetric group acting on the set $\Omega=\{1,2,\ldots\}$. The $k$-valued measures for $\fS$ are parametrized by elements of $\bZ_p$ \cite[\S 15.7]{repst}. For $t \in \bZ_p$, let $\mu_t$ denote the corresponding measure. This measure is not normal, but it is known that $\uPerm(\fS; \mu_t)$ admits a weak abelian envelope $\cC_t$. The category $\cC_t$ can be constructed using ultraproducts; see \cite[\S 3.3]{Harman1}. It is reasonable to expect that $\cC_t$ is the abelian envelope of $\uPerm(\fS; \mu_t)$, but this not known.

Let $\Omega^{[n]}$ denote the subset of $\Omega^n$ consisting of points with distinct co-ordinates. Let $\bS(\fS; \sE)$ denote the full subcategory of $\bS(\fS)$ consisting of $\fS$-sets that decompose into orbits of the form $\Omega^{[n]}$. It turns out that this category is closed under products and fiber products, and is therefore itself a $\rB$-category (though it is not pre-Galois, as equivalence relations are not effective). In the terminology of \cite[\S 2.6]{repst}, $\bS(\fS; \sE)$ is the category of $\sE$-smooth $\fS$-sets, where $\sE$ is the stabilizer class generated by $\Omega$.

There is a notion of measure for the category $\bS(\fS; \sE)$; see \cite[Remark~3.2]{repst}. It turns out that such measures are parametrized by elements of $k$; see \cite[Remark~15.8]{repst}. Let $\nu_a$ be the measure corresponding to $a \in k$. If $a \in \bF_p$ then $\uPerm(\fS, \sE; \nu_a)$ embeds into $\cC_t$ for any $t \in \bZ_p$ lifting $a$, and thus maps to $\cC_t$. It is not difficult to see that all such $\cC_t$ are weak abelian envelopes of $\uPerm(\fS, \sE; \nu_a)$. The existence of these varied weak abelian envelopes in fact implies that $\uPerm(\fS, \sE; \nu_a)$ does not admit an abelian envelope. For $a \not\in \bF_p$, the category $\uPerm(\fS, \sE; \nu_a)$ does not even admit a weak abelian envelope, as it has a nilpotent endomorphism with non-zero trace.
\end{example}

\subsection{The main theorem: the discrete case}

Recall that a pre-Tannakian category is \defn{discrete} if every object is a subquotient of an \'etale algebra. We can now establish our rough classification of such categories, which is the first half of Theorem~\ref{mainthm}.

\begin{theorem}
Let $\cC$ be a discrete pre-Tannakian category, let $G=\pi^{\rm olig}(\cC)$, and let $\mu$ be the induced measure on $G$. Then the natural functor $\Phi \colon \uPerm(G; \mu) \to \cC$ is a weak abelian envelope.
\end{theorem}

\begin{proof}
The measure $\mu$ and the functor $\Phi$ are constructed in Theorem~\ref{thm:meas}, which also shows that $\Phi$ is fully faithful. Since every \'etale algebra is in the essential image of $\Phi$ and every object of $\cC$ is a subquotient of an \'etale algebra, it follows that $\Phi$ is a weak abelian envelope.
\end{proof}

\begin{remark} \label{rmk:olig}
Suppose that in the setting of the theorem $\cC$ is generated by a single \'etale algebra $A$. Let $\cB$ be the sub pre-Galois category of $\bS(\cC)$ generated by $\Sp(A)$; since $\cB$ is finitely generated, it has the form $\bS(H)$ for an oligomorphic group $H$. Mimicking the proof of the theorem, one finds that $\cC$ is the weak abelian envelope of $\uPerm(H; \nu)$ for some measure $\nu$ on $H$. Thus for finitely generated pre-Tannakian categories, one can use oligomorphic groups in place of admissible groups.
\end{remark}

\subsection{Strongly discrete pre-Tannakian  categories}

Recall that a pre-Tannakian category is \defn{strongly discrete} if every object is a quotient of an \'etale algebra. We now establish some properties of such categories.

\begin{proposition} \label{prop:strong-proj}
If $\cC$ is discrete and has enough projectives then it is strongly discrete.
\end{proposition}

\begin{proof}
Let $X$ be an arbitrary object of $\cC$. Since $\cC$ has enough projectives, there is a surjection $P \to X$ with $P$ projective. Since $\cC$ is pre-Tannakian, $P$ is also injective \cite[Proposition~6.1.3]{EGNO}. Since $\cC$ is discrete, $P$ is a subquotient of some \'etale algebra $A$. Since $P$ is injective and projective, it is a summand of $A$, and in particular a quotient of $A$. Thus $X$ is a quotient of $A$ and so $\cC$ is strongly discrete.
\end{proof}

\begin{proposition} \label{prop:strong}
Suppose that $\cC$ is discrete. Then $\cC$ is strongly discrete if and only if every exact sequence in $\cC$ splits after an \'etale base change.
\end{proposition}

\begin{proof}
Suppose $\cC$ is discrete. Consider a surjection $f \colon A \to \bbone$ in $\cC$, where $A$ is an \'etale algebra. We claim that $f$ is \'etale locally split. We have a decomposition $A=\bigoplus_{i=1}^n A_i$ where each $A_i$ is atomic. Let $f_i=f \vert_A$. Since $\bbone$ is simple, some $f_i$ is surjective, and it suffices to show that $f_i$ is \'etale locally split. Thus, relabeling, it suffices to treat the case where $A$ is atomic. The $k$-vector space $\Hom_{\cC}(A, \bbone)$ is one-dimensional and spanned by $\epsilon$, and so $f=a \cdot \epsilon$ for some non-zero $a \in k$. Consider the map $\id \otimes f \colon A \otimes A \to A$. This is split by $(1 \otimes a^{-1}) \cdot \delta$, and so the claim is established. 

Next, consider a surjection $f \colon X \to \bbone$, where $X$ is any object of $\cC$. Since $\cC$ is strongly discrete, there is a surjection $g \colon A \to X$ for some \'etale algebra $A$. By the previous paragraph, there is some \'etale algebra $B$ such that $\id \otimes fg \colon B \otimes A \to B$ is split. Letting $s \colon B \to B \otimes A$ be a splitting, we have $(\id \otimes f) s'=\id_B$, where $s'=(\id \otimes g) \circ s$. Thus $f$ is \'etale locally split.

The general case now follows by the same method in \cite[Lemma~7.14]{Deligne1}. Precisely, let $f \colon X \to Y$ be a surjection in $\cC$. There is induced surjection $\uHom(Y, X) \to \uHom(Y, Y)$ and a natural map $\alpha \colon \bbone \to \uHom(Y, Y)$, namely, co-evaluation. Let $X'$ be the fiber product of $\uHom(Y,X)$ and $\bbone$ over $\uHom(Y,Y)$. By the previous paragraph, the surjection $X' \to \bbone$ is \'etale locally split. This yields a local splitting of $f$, as required.

Now suppose that sequences split after an \'etale base change. Let $X$ be a given object in $\cC$. Since $\cC$ is discrete, there is a surjection $Y \to X$ where $Y$ is a subobject of some \'etale algebra $A$. By assumption, the inclusion $Y \to A$ is locally split, and so there is some \'etale algebra $B$ and a splitting $B \otimes A \to B \otimes Y$, which is necessarily surjective. Composing with the surjection $B \otimes Y \to B \otimes X$, we see that $X$ is a quotient of the \'etale algebra $B \otimes A$, as required.
\end{proof}

\begin{remark}
Deligne \cite[Lemma~7.14]{Deligne1} proved that in any pre-Tannakian category in characteristic~0, a short exact sequence splits after an fppf base change.
\end{remark}

To close this discussion, we give an example to show that not all discrete pre-Tannakian categories are strongly discrete.

\begin{example} \label{ex:not-strong}
Suppose $k$ has characteristic $p$, and for $t \in \bZ_p$ let $\cC_t$ be the category defined in Example~\ref{ex:weak}. We claim that $\cC_t$ is discrete but not strongly discrete. Let $\tilde{\cC}_t$ be the ultraproduct of the category $\Rep_k(\fS_n)$, taken with respect to an ultrafilter in which the sequence $1,2,\ldots$ converges to $t$ in $\bZ_p$. Let $X_n=k^n$ be the permutation representation of $\fS_n$, and let $X$ be the ultralimit of the $X_n$'s in $\tilde{\cC}_t$. Then $\cC_t$ is (by definition) the tensor subcategory of $\tilde{\cC}_t$ generated by $X$. Since $X$ is an \'etale algebra, this shows that $\cC_t$ is discrete.

Now, let $X_n^{[p]}$ be the permutation representation of $\fS_n$ with basis indexed by $p$-tuples in $\{1,\ldots,n\}$ with distinct co-ordinates, and let $X^{[p]}$ be the ultralimit. Similarly define $X^{(p)}_n$ and $X^{(p)}$, but using $p$-element subsets. There is a natural surjection $\pi \colon X^{[p]} \to X^{(p)}$, and a unique copy of the trivial representation $\bbone$ in $X^{(p)}$. Let $Y=\pi^{-1}(\bbone)$. One can show that $Y$ is not a quotient of an \'etale algebra in $\cC_t$, which shows that $\cC_t$ is not strongly discrete.

Let us explain the basic idea when $p=2$. The invariant element of $X^{(2)}_n$ is the sum of all 2-element subsets of $\{1,\ldots,n\}$. An inverse image of this element under $\pi$ amounts to choosing an ordering on each such 2-element set. There is no ``reasonable'' way to do this; the stabilizer of such a point is contained in a Young subgroup of the form $\fS_{n/2} \times \fS_{n/2}$, and these groups do not converge to an open subgroup of $\fS$. This prevents one from realizing $Y$ as the image of an \'etale algebra, as such an algebra is a quotient of sums of $X^{[r]}$'s.
\end{example}

\subsection{The main theorem: the strongly discrete case}

We now prove the second half of Theorem~\ref{mainthm}, which gives a very precise classification of strongly discrete pre-Tannakian categories.

\begin{theorem} \label{thm:strong}
Let $\cC$ be a strongly discrete pre-Tannakian category, let $G=\pi^{\rm olig}(\cC)$, and let $\mu$ be the induced measure on $G$.
\begin{enumerate}
\item The measure $\mu$ is normal; if $\cC$ is semi-simple then it is regular.
\item The category $\cC$ is the abelian envelope of $\uPerm(G; \mu)$.
\item $\cC$ is equivalent to $\uRep^{\rf}(G; \mu)$, and any object of $\uRep(G; \mu)$ is the union of its finite length subobjects.
\end{enumerate}
\end{theorem}

An analog to Remark~\ref{rmk:olig} applies to this theorem. The theorem implies that if $\cC$ is semi-simple discrete pre-Tannakian category then $\pi^{\rm olig}(G)$ must carry a regular measure. We can use this to show that certain kinds of categories do not exist. For instance:

\begin{corollary}
Let $G$ be the oligomorphic group of orientation preserving self-homeo\-morphisms of the unit circle. There is no semi-simple discrete pre-Tannakian category $\cC$ such that $\pi^{\rm olig}(\cC)=G$.
\end{corollary}

\begin{proof}
We showed in \cite[\S 17.9]{repst} that $G$ has no regular measure.
\end{proof}

\begin{remark}
There is a (non-semi-simple) pre-Tannakian category associated to the group $G$ in the corollary; see \cite{circle}.
\end{remark}

We prove Theorem~\ref{thm:strong} in the next several lemmas. We let $G=\pi^{\rm olig}(\cC)$ and identify $\bS(\cC)$ with $\bS(G)$. We use notation similar to \S \ref{s:meas}, e.g., for a finitary $G$-set $X$, we write $k[X]$ for the corresponding \'etale algebra in $\cC$. We let
\begin{displaymath}
\Phi \colon \uPerm(G; \mu) \to \cC
\end{displaymath}
be the functor produced by Theorem~\ref{thm:meas}.

\begin{lemma}
The measure $\mu$ is normal.
\end{lemma}

\begin{proof}
Let $f \colon X \to Y$ be a surjection of finitary $G$-sets. Since every atom in $Y$ is hit, it follows from the description of ideals of $k[Y]$ (Corollary~\ref{cor:rigid-mod-1}) that $f^* \colon k[Y] \to k[X]$ is injective; thus $f_* \colon k[X] \to k[Y]$ is surjective. Since $\cC$ is strongly discrete, $f_*$ is \'etale locally split. Thus there is some open subgroup $U$ such that $\id \otimes f_* \colon k[G/U] \otimes k[X] \to k[G/U] \otimes k[Y]$ has a $k[G/U]$-linear splitting. Since $\Phi$ is fully faithful and takes $A_f$ to $f_*$, it follows that $\id \otimes A_f \colon \Vec_{G/U} \uotimes \Vec_X \to \Vec_{G/U} \uotimes \Vec_Y$ has a $\Vec_{G/U}$-linear splitting. One easily sees that such a splitting corresponds to a splitting of $A_f \colon \Vec_X \to \Vec_Y$ in the category $\uPerm(U; \mu)$. Thus $A_f$ has a $U$-equivariant right inverse, and so $\mu$ is normal by \cite[Proposition~11.7]{repst}.
\end{proof}

\begin{lemma}
If $\cC$ is semi-simple then $\mu$ is regular.
\end{lemma}

\begin{proof}
Let $X$ be a transitive $G$-set. To show that $\mu$ is regular, we must show $\mu(X)$ is non-zero. Let $f \colon X \to \bone$ be the map to the one point $G$-set. We thus have a surjection $f_* \colon k[X] \to \bbone$ in $\cC$. Since $\cC$ is semi-simple, the functor $\Gamma$ is exact, and so $f_* \colon \Gamma(k[X]) \to \Gamma(\bbone)$ is surjective. It follows that $\mu(f)=f_*(1)$ is non-zero. Since $\mu(X)=\mu(f)$, we thus see that $\mu$ is regular.
\end{proof}

\begin{lemma}
$\cC$ is the abelian envelope of $\uPerm(G; \mu)$ via $\Phi$.
\end{lemma}

\begin{proof}
According to \cite[Theorem~3.1.4]{CEAH}, it suffices to verify the following conditions:
\begin{enumerate}[(i)]
\item $\Phi$ is fully faithful.
\item Any object of $\cC$ is isomorphic to the image of some map $\Phi(\Vec_X) \to \Phi(\Vec_Y)$.
\item For any epimorphism $f \colon M \to N$ in $\cC$, there exists a finitary $G$-set $X$ such that $\id \otimes f \colon \Phi(\Vec_X) \otimes M \to \Phi(\Vec_X) \otimes N$ splits.
\end{enumerate}
We have already shown (i) in Theorem~\ref{thm:meas}. Condition~(ii) follows since $\cC$ is strongly discrete: given $M \in \cC$, we have surjections $\Phi(\Vec_X) \to M$ and $\Phi(\Vec_Y^{\vee}) \to M^{\vee}$, for some finitary $G$-sets $X$ and $Y$, and so $M$ is the image of the map $\Phi(\Vec_X) \to \Phi(\Vec_Y)$. Condition~(iii) follows from Proposition~\ref{prop:strong}.
\end{proof}

\begin{lemma} \label{lem:strong-3}
$\cC$ is equivalent to $\uRep^{\rf}(G; \mu)$, and any object of $\uRep(G; \mu)$ is the union of its finite length subobjects.
\end{lemma}

\begin{proof}
The proof of involves verifying a large number of details. In the interest of space, we just provide an outline of the main ideas. We proceed as follows.

(a) For an object $M$ of $\cC$ and an open subgroup $U$ of $G$, we define the ``$U$-invariants'' $M^U$ of $M$ to be $\Hom_{\cC}(k[G/U], M)$. This is a finite dimensional $k$-vector space. Define $\Phi(M)$ to be the direct limit of the $M^U$ over open subgroups $U$. This is a (typically infinite dimensional) $k$-vector space. The functor $\Phi$ is obviously left-exact. Using the strong discreteness of $\cC$ (in the form of \'etale local splittings), one shows that $\Phi$ is exact.

(b) We define on $\Phi(M)$ the structure of a smooth $A$-module, where $A=A(G)$ is the completed group algebra. The basic idea is as follows. Suppose $x \in M^U$ and $a \in A$, and let $a_U \in \cC(G/U)$ be the $U$-component of $a$. (Recall that $A$ is the inverse limit of the $\cC(G/U)$'s.) Suppose that $a_U$ is left $V$-invariant. Then $a_U$ corresponds to a morphism $k[G/V] \to k[G/U]$ in $\cC$, and we define $ax$ to be the composition of this map with $x \colon k[G/U] \to M$. Thus $ax$ is an element of $M^V$.

(c) We now verify that for $M=k[G/U]$, the $A$-module $\Phi(M)$ is naturally isomorphic to $\cC(G/U)$, with its standard $A$-module structure. Using this, one shows that $\Phi$ is fully faithful; the idea is to choose presentations by \'etale algebras (which is possible by strong discreteness) to reduce to the $k[G/U]$ case.

(d) We next show that $\Phi$ maps simple object of $\cC$ to simple object of $\uRep(G; \mu)$. The idea is that if $S$ is a simple of $\cC$ and $\Phi(S)$ is not simple, then there is some map $\cC(G/U) \to \Phi(S)$ that detects a proper non-zero subobject of $S$, and this would yield such a map $k[G/U] \to S$ back in $\cC$, which is a contraction. From this, $\Phi$ preserves finite length objects. Since $\cC(G/U)$ is in the essential image of $\Phi$, it has finite length. As these objects are generators, it follows that every object of $\uRep(G; \mu)$ is the union of its finite length subobjects. It also follows that the essential image of $\Phi$ consists of all finite length objects, and so $\Phi$ induces an equivalence $\cC \to \Rep^{\rf}(G; \mu)$.

(e) Finally, we show that $\Phi$ naturally has the structure of a symmetric tensor functor. Let $M$ and $N$ be objects of $\cC$. Then there is a natural $k$-bilinear map
\begin{displaymath}
\Phi(M) \times \Phi(N) \to \Phi(M \otimes N).
\end{displaymath}
One verifies that this is strongly bilinear in the sense of \cite[\S 12.2]{repst}. It thus follows from the definition of the tensor product $\uotimes$ on $\uRep(G; \mu)$ that there is an induced map
\begin{displaymath}
\Phi(M) \uotimes \Phi(N) \to \Phi(M \otimes N).
\end{displaymath}
One then checks that this is an isomorphism when $M$ and $N$ have the form $k[G/U]$. The general case then follows by choosing presentations by such objects (which uses strong discreteness).
\end{proof}

\begin{remark} \label{rmk:sheaf}
There is an alternative approach to proving Theorem~\ref{thm:strong}(c). Since $\mu$ is a normal measure, the category $\uPerm(G; \mu)$ carries a natural linear Grothendieck topology \cite[\S 14.2]{repst}. It turns out that $\uRep(G; \mu)$ is exactly the category of sheaves for this topology \cite[Theorem~14.1]{repst}. It is not difficult to see that objects of $\cC$ also represent sheaves, and this yields a functor $\cC \to \uRep(G; \mu)$. One then verifies the necessary properties of this functor.
\end{remark}

\end{document}